\documentclass[11pt,a4paper]{article}
\usepackage{amssymb,amsmath, amsfonts}
\usepackage{graphicx,graphics}
\usepackage{mathtools}
\usepackage{bbold}
\usepackage[english]{babel}
\usepackage[utf8]{inputenc}
\usepackage{epsfig,url}
\usepackage{bbm,theorem}
\usepackage{a4wide}
\usepackage{xcolor}
\usepackage{enumerate}
\usepackage{calrsfs}
\usepackage[bookmarks=true]{hyperref}
\usepackage{bookmark}
\usepackage{verbatim}
\DeclareMathAlphabet{\pazocal}{OMS}{zplm}{m}{n}
\newtheorem{theorem}{Theorem}[section]
\newtheorem{definition}[theorem]{Definition}
\newtheorem{corollary}[theorem]{Corollary}

\newtheorem{proposition}[theorem]{Proposition}
\newtheorem{assumption}[theorem]{Assumption}

\newtheorem{remark}[theorem]{Remark}

\numberwithin{equation}{section}
\numberwithin{theorem}{section}
\newcommand{\qed}{\hfill$\Box$}
\newenvironment{proof}{\begin{trivlist}\item[]{\em Proof:}\/}{\qed\end{trivlist}}
\newenvironment{proofof}[1][Proof]{\noindent \textit{#1:} }{\ \qed}

\newcommand{\R}{{\mathbb R}}                                                    
\newcommand{\sd}{{\R_*^d}}
                                                    
\newcommand{\N}{{\mathbb N}}                                                    
\newcommand{\dsd}{ { \N_0^d\setminus\{0\} } }

\newcommand{\M}{{\mathcal M}}

\newcommand{\X}{{\mathcal{X}}}

\newcommand{\op}{{\mathcal{T} }}                               
\newcommand{\opi}{{\mathcal{T}_1 }}                            
\newcommand{\opii}{{\mathcal{T}_2 }}                               
\newcommand{\xt}{X_{\feo,T,\epsilon}}

\newcommand{\trunc}{{\{\epsilon\leq|x|\leq2/\epsilon\}}}

\newcommand{\fe}{{f_\epsilon}}
\newcommand{\fen}{{f_{\epsilon_n}}}
\newcommand{\feo}{{f_{\epsilon,0}}}

\newcommand{\ggel}{\gamma_\text{gel}}
\newcommand{\pgel}{\lambda_\text{gel}}

\DeclareMathOperator*{\supp}{supp} 

\newcommand{\ind}{\mathbb{1}}

\usepackage{enumitem}

\begin{document}

\title{Global existence  of measure-valued solutions to the multicomponent Smoluchowski coagulation equation}

\newcommand \ale{\color{purple}}
\newcommand \mf{\color{blue}}
\newcommand \rd{\color{blue}}

\newcommand{\email}[1]{E-mail: \tt #1}

\newcommand{\emailmarina}{\email{marina.ferreira@math.univ-toulouse.fr}}
\newcommand{\UTaddress}{\em CNRS $\&$ Institut de Mathématiques de Toulouse, UMR 5219, \\  \em
Université Paul Sabatier, 118 route de Narbonne,
31062 Toulouse Cedex 9, France 
 }

\newcommand{\emailsakari}{\email{sakari.pirnes@helsinki.fi}}
\newcommand{\UHaddress}{\em University of Helsinki, Department of Mathematics 
and Statistics\\
\em P.O. Box 68, FI-00014 Helsingin yliopisto, Finland}

\author{Marina A. Ferreira \thanks{\emailmarina}, Sakari Pirnes \thanks{\emailsakari}
\\[1em]%
$\,^*$\UTaddress
\\[1em]%
$\,^\dag$\UHaddress}

\maketitle

\begin{abstract}

Global solutions to the multicomponent Smoluchowski coagulation equation are constructed for measure-valued initial data with minimal assumptions on the moments. The framework is based on an abstract formulation of the Arzelà-Ascoli theorem for uniform spaces. The result holds for a large class of coagulation rate kernels, satisfying a power-law upper bound with possibly different singularities at small-small, small-large and large-large coalescence pairs. This includes in particular both mass-conserving and gelling kernels, as well as interpolation kernels used in applications. We also provide short proofs of mass-conservation and gelation results for any weak solution, which extends previous results for one-component systems.

\end{abstract}

\bigskip


\bigskip


\newpage
\tableofcontents

%auto-ignore
% final sec 1
\section{Introduction}
\subsection{Aim of the paper}
\label{sec:aim}
In this paper, we will prove existence, mass-conservation and gelation results for a large
class of Smoluchowski coagulation equations describing multicomponent systems.
 In coagulation
systems the particles are often characterized by one parameter representing
their size. In the discrete multicomponent system, the clusters are characterized 
by a $d$-dimensional composition vector
$\alpha=(\alpha_1,\alpha_2,\dots,\alpha_d)\in \dsd$,
where $\N_0=\{0,1,2,\dots\}$.
The components of the composition vector can characterize for instance the
number of molecules of each chemical component of the particles or their
parameters such as charge (cf.\ \cite{Vehkam}). In the continuous
case, the composition vector $x\in\sd=[0,\infty)^d\setminus\{0\}$
allows the parameters of the particles to take continuous values, e.g. the
volume of each chemical component. 

The time evolution of the
composition distributions $n$ and $f$ in the discrete and continuous cases,
respectively, are given by the coagulation equations
\begin{subequations}
\label{B4}%
    \begin{align}
        \label{eq:B4_discrete}%
        \partial_t n(\alpha,t)
        &=
        \frac{1}{2}
        \sum_{\beta<\alpha}
        K(\alpha-\beta,\beta)n(\beta,t)n({\alpha-\beta},t)
        -
        n(\alpha,t)\sum_{\beta>0} K({\alpha,\beta})n(\beta,t),
        \\
        \label{eq:B4_continuous}%
        \partial_t f(x,t)
        &=
        \frac{1}{2}
        \int_{\{0<y<x\}}
        K(x-y,y)f(x-y,t)f(y,t)dy
        -
        f(x,t)
        \int_{\sd}
        K(x,y)f(y,t)dy,
    \end{align}%
\end{subequations}%
where $t\ge0$ represents time and $K$ is the collision kernel which is assumed
to be non-negative and symmetric. The notation $a<b$ used for vectors $a,b\in \R^d$
means that $a\ne b$ and $a\leq b$, where $a\leq b$ means that $a_k\leq b_k$ for
all $k=1,2,\dots,d$.

In order to study simultaneously the discrete \eqref{eq:B4_discrete} and
continuous  \eqref{eq:B4_continuous} equations we will be working with measure-valued solutions $f(dx,t)$. 
This will
include the continuous case \eqref{eq:B4_continuous} when $f(dx,t)=f(x,t)dx$,
where we have abused the notation by denoting the density of the measure
$f(dx,t)$ with respect to the Lebesgue measure $dx$ at time $t$  by $f(x,t)$.
And the discrete case \eqref{eq:B4_discrete} when the measure-valued
solution has the form
\begin{align}
    \label{eq:dis_is_cont}
    f(dx,t) = \sum_{\alpha\in\N_0^d\setminus\{0\}} n_\alpha(t)\delta_\alpha(dx),
\end{align}
where $\delta_\alpha$ is the Dirac measure supported at $\alpha$.
The precise definition of measure-valued solutions will be given in the next
section in Definition \ref{def:time-dep-sol-thesis}.

Most of the previous papers in mathematics consider one-component 
systems, altough in many applications the systems have several components. Multicomponent systems with $d=2$ have been introduced in \cite{L} by Lushnikov, where solutions to the continuous equation \eqref{eq:B4_continuous} with the constant kernel $K(x,y)=1$ were computed for exponentially distributed initial data. More recently,  formulas of solutions to \eqref{eq:B4_discrete} and \eqref{eq:B4_continuous} for  other explicitly solvable kernels, namely, the additive $K(x,y)=|x+y|$ and multiplicative $K(x,y)=|x||y|$ kernels were obtained \cite{FDGG, FDGG2, Krapivsky} with initial data supported on the monomers or following a gamma distribution. The procedures developed in these papers to obtain explicit solutions rely on multidimensional Laplace transform methods, which generalizes the available methods to solve one-component equations; see, for instance, \cite{MP04}.
Interestingly, in multicomponent systems, there are classes of kernels for which Laplace methods do not apply, but for which an explicit solution can be obtained by assuming certain symmetries on the multicomponent kernels and by using the explicit solutions available for one-component systems. This is the case for kernels which are constant along lines passing through the origin, i.e., satisfying $K(r\theta,s\theta) = Q(\theta),$ with $\theta \in \R^d, |\theta|=1,\ r,s>0$ \cite{localization}.

However, for most kernels used in applications, explicit formulas for the solutions are not available. This is the case in atmospheric aerosol science, where the diffusion and ballistic
kernels are used \cite{Olenius, Vehkam} (see also \cite{Fried, Fuchs} where the physical aspects of aerosol particles are described). In the multicomponent case, the diffusion and ballistic kernels are respectively of the form \cite{FLNV2}
\begin{align}
    \label{eq:diffusion}
    K(x,y) = c_0
                \left( |x|^{-\frac{1}{3}}+ |y|^{-\frac{1}{3}}  \right) 
                \left( |x|^{\frac{1}{3}}+ |y|^{\frac{1}{3}}  \right) 
\end{align}
and 
\begin{align}
    \label{eq:ballistic}
                K(x,y) = c_0
                \left( |x|^{-1}+ |y|^{-1}  \right)^\frac{1}{2} 
                \left( |x|^{\frac{1}{3}}+ |y|^{\frac{1}{3}}  \right) ^2.
\end{align}

Here, $c_0>0$ and  $|\cdot|$ denotes the $\ell^1$-norm of $\R^d$. The diffusion
 kernel \eqref{eq:diffusion} is used to describe the coagulation rate between particles with
radius of order one micrometre (the so-called continuum region), while the ballistic
kernel \eqref{eq:ballistic} is used for particles with radius of order one nanometre (free molecular region).
Up to today, the rate between other pairs of particles remains unclear, including the rates between medium-sized particles, and  between  small and  large particles
(see \cite{Veshchunov}, \cite{Suresh} and \cite{Thajudeen}, where several derivations of such rates are discussed using numerical simulations, scaling analysis and measurements).
The insufficient data on coagulation rates can be overtaken by considering general classes of kernels which allow for possibly different asymptotic behaviours in different regions of the size space.

In this paper we then consider a very general class of coagulation rate kernels. 
We suppose that the coagulation rate kernel $K$ is continuous, non-negative and 
symmetric, namely,
\begin{align}
    \label{eq:condK_cont}
    K\in C(\sd\times \sd;[0,\infty)), \text{ and } K(x,y)=K(y,x), \ \forall
    x,y\in \sd.
\end{align}
Moreover, we suppose that the kernel $K$ satisfies the following upper bound

\begin{align}
    \label{eq:condK_sym2}
    K(x,y) \leq c_2
    \begin{cases}
         |x|^{-\beta}|y|^{-\beta},&  |x|,|y|\leq1,\\
         |x|^{\gamma_1+\lambda_1}|y|^{-\lambda_1},&  |x|\ge1,|y|\leq1,\\
         |x|^{\gamma_2+\lambda_2}|y|^{-\lambda_2},&  |x|,|y|\ge1 ,\\
    \end{cases}
    \quad
    \text{for }|y|\leq |x|,
    \end{align}
for a constant $0<c_2<\infty$ and parameters $\beta\in\R$ and $\gamma,\lambda\in \R^2$
satisfying $-\lambda_j\leq\gamma_j+\lambda_j$  and $\gamma_j+\lambda_j\leq 1$ for {$j=1,2$}. 
Note that the assumption  $-\lambda_j\leq\gamma_j+\lambda_j$ can be assumed without loss of generality. The reason we do not have two different exponents when $|x|,|y|\leq 1$ is that we could always choose $-\beta$ to be the smaller of the two exponents without losing any relevant information on the kernel or initial data.

While our mass-conservation and gelation results do not need any additional conditions on the kernel, due to technical reasons we have to assume $\gamma_j+\lambda_j<1$ for both $j=1,2$ in order to prove existence.
The existence result hold to any initial data $f_0$ that is a positive Radon measure and satisfies 
\begin{equation}
       \label{eq:ini_cond}%
        \int_{\sd}\left(|x|^{\min(-\beta,-\lambda_1)-r}+|x|\right)f_0(dx) <\infty
\end{equation} 
for some small $r>0$, and the mass-conservation and gelation results hold with $r=0$.

The upper bound \eqref{eq:condK_sym2} generalizes many classes of kernels in the literature which are bounded by only one  power law defined by
$K_{\lambda',\gamma'}(x,y)=|x|^{-\lambda'}|y|^{\gamma'+\lambda'}+|y|^{-\lambda'}|x|^{\gamma'+\lambda'}$
with real parameters $\gamma',\lambda'\in\R,-\lambda'\leq\gamma'+\lambda'$ (see for instance \cite{Dust, FL06, localization, ThromUniqueness}).
We recall that the parameter $\gamma'$ characterizes the rate of coagulation of particles of similar  sizes and $\lambda'$ fine tunes the rate of coagulation of particles of different sizes.
The parameter $\gamma'$ is called homogeneity of $K_{\lambda',\gamma'}(x,y)$ in the literature.
Previous results  (see for instance \cite{LuisaIyerMagnanini23, Barik} and the book \cite{mathbook}) show that $\gamma'\leq1$ is sufficient to ensure mass-conservation, whereas $\gamma'>1$ leads to gelation, i.e., loss of mass-conservation in finite time.

One reason for considering the  upper bound \eqref{eq:condK_sym2} is that 
in this paper we prove that for mass to be conserved it is sufficient to only assume that $\gamma_2\leq1$ and no extra assumptions on the other parameters $\beta,\lambda$ and $\gamma_1$ are needed. 
This means that a kernel can have homogeneity greater than $1$ for 
large-small coalescence pairs and still conserve mass. 
This generalizes the state of the art results in the literature also in the one-component case \cite{Barik}.
Another reason is that this class contains kernels  which satisfy different power-law behaviours 
at small-small, large-large and small-large coalescence pairs, such as
the so-called transition kernels for aerosol particles which equal the diffusive kernel \eqref{eq:diffusion} for $|x|,|y| \geq 1$ 
%in the continuum region (large sizes) 
and the ballistic kernel \eqref{eq:ballistic} for $|x|,|y| \leq 1$,
%in the free molecular region (small sizes) 
i.e., they satisfy \eqref{eq:condK_sym2} with $-\beta=-1/2, \ \gamma_2=0,\ -\lambda_2=-1/3$ and for some $\gamma_1, \lambda_1 \in \R$. 

For completion, we follow the proof given in \cite{Dust} for $d=1$ to show that any given solution fails to conserve mass, i.e., gelation occurs in finite time, provided the kernel is bounded below by a polynomial with homogeneity greater than 1. To this end, we will assume  in addition to \eqref{eq:condK_cont} and \eqref{eq:condK_sym2} the lower bound 
\begin{align}
    \label{eq:lower}
    c_1(|x|^{\ggel+\pgel}|y|^{-\pgel}+|y|^{\ggel+\pgel}|x|^{-\pgel})\leq K(x,y) \quad \forall x,y\in\sd
\end{align}    
for some $c_1>0$, $\ggel>1$ and $-\pgel\leq \ggel+\pgel$. Gelling solutions for multicomponent equations have been recently constructed and studied in \cite{ Heydecker2019, Hoogendijk2024, cross-multi,Norris2000} for multiplicative kernels of the form $K(x,y) = x^TAy$ with $A$ symmetric and non-negative matrix, i.e., they satisfy \eqref{eq:lower} with $\ggel = 2$, and for general gelling kernels in \cite{LuisaIyerMagnanini23} for a generalization of the Smoluchowski coagulation equation, called the Flory's equation. 

Previous well-posedness results for $d=1$ have been obtained for the discrete equation \eqref{eq:B4_discrete} (see for instance \cite{Ball,  Jeon, McLeod}), as well as for the continuous equation \eqref{eq:B4_continuous} with initial data in $L^1$, see for example \cite{Barik, Camejo, Dubowski, Dust, EM05} and the recent book by Lamb, Laurençot and Banasiak \cite{mathbook}. In particular, the recent existence result of mass-conserving solutions in $L^1$ obtained by Barik, Giri and Laurençot in \cite{Barik} holds for kernels satisfying \eqref{eq:condK_sym2} with $\beta>0$, $-\lambda_1=-\beta$ and $\gamma_1+\lambda_1 = \gamma_2+\lambda_2 = 1$ and $-\lambda_2 = 0$ and initial condition satisfying $\int_0^\infty  (x^{-2\beta}+ x) f_0(x)dx<\infty$. We allow more singular initial data by only assuming finiteness of the $-\beta-r$ moment in \eqref{eq:ini_cond}. Although we prove mass-conservation for a strictly larger class of kernels, we are not able to prove existence for kernels that require $\gamma_j+\lambda_j=1$ for $j=1$ or $j=2$ due to the technical assumption $\gamma_j+\lambda_j<1$ needed by our existence proof. On the other hand, we prove existence for mass-conserving kernels with $-\lambda_2\ne0$ or $-\lambda_1\ne-\beta$, and the same existence proof holds for a large class of gelling kernels.

Studies considering measure-valued solutions are more scarce  in the literature. 
In the case of one-component systems, Menon and Pego \cite{MP04} prove existence and uniqueness of measure-valued solutions  for the explicitly solvable kernels. 
 Norris  \cite{Norris} and  Fournier and Laurençot \cite{FL06} obtain several well-posedness results for general kernels that satisfy $K(x,y)\leq\omega(x)\omega(y)$, where $\omega$ is a sub{linear} function, i.e., it satisfies $\omega(a x) \leq a \omega(x)$ for $a \geq 1$, $x >0$.
 In particular, under stronger condition, $\int_{(0,\infty)}  (\omega(x) + \omega(x)^2) f_0(dx)<\infty$, on the initial data and an additional assumption, $\omega^2$ is sublinear or $K(x,y)\leq\omega(x)+\omega(y)$, on the kernel, a global existence result is given in \cite[Theorem 2.1]{Norris} and generalized in \cite{Norris2000} for multi-component systems.
  Interestingly, by using a family of  stochastic coalescent processes, the conditions on the initial data can be relaxed to just one moment bound $\int_{(0,\infty)}  \omega(x) f_0(dx)<\infty$ and the additional condition on the kernel to 
  \begin{align}
      \label{eq:Norris_cond}
       K(x,y)\omega(x)^{-1}\omega(y)^{-1}\to 0 \quad\text{as}\quad (x,y)\to\infty,
  \end{align}
   yielding the most general global existence result for one-component measure-valued solutions we are aware of (see Theorem 4.1 in  \cite{Norris}).
By relying on topological tools in uniform spaces, we are able to develop compactness arguments in the space of time-continuous measures, which allow us to construct a global solution under very minimal assumptions on the initial data and on the kernel. Thus, in addition to extending the existence result in \cite[Theorem 4.1]{Norris} by removing some technical conditions  on the kernel, including \eqref{eq:Norris_cond}, and by allowing for vector-valued size variables, our techniques also allow us to relax the assumptions on the initial data used in \cite{Norris2000},
 only requiring the condition \eqref{eq:ini_cond}. 
Our proof also gives global existence for gelling kernels, extending the results in \cite{Norris2000}, where existence was obtained only up until the gelation time. 
As a Corollary we also obtain existence of time differentiable strong solutions to the discrete coagulation equation \eqref{eq:B4_discrete}.

Moreover, Norris proves mass-conservation for solutions with finite $\omega^2$ moment in \cite{Norris,Norris2000} if there exists $\epsilon>0$ such that $\omega(x)\ge\epsilon x$ for all $x>0$, and so, obtains existence of mass-conserving (global) measure-valued solutions  with the additional assumption $K(x,y)\leq\omega(x)+\omega(y)$ on the kernel. Therefore, our existence and mass-conservation results hold for a strictly larger class of kernels and initial data.

Finally, we remark that
multicomponent systems contain also mathematically interesting phenomena which
are not encountered in the one-component case. The localization considered
in \cite{localization} is an example of such phenomenon. See also the survey \cite{chp} where localization is shown in explicit solutions. 
Our results imply the existence of mass-conserving solutions whose construction had only been outlined in \cite{localization}.
Combining the uniqueness result in \cite{ThromUniqueness} and the time estimates in \cite{localization} with our results further completes the picture of well-posedness and long-time localization for the multicomponent Smoluchowski coagulation equation for a general kernel and initial data (see Section \ref{sec:localization} for the precise statement and class of kernels).

\subsection{Notations}

We collect here the main notations used in this paper. 

We denote size space of particle compositions by $\sd =
[0,\infty)^d\setminus\{(0,\dots,0)\}$. We will use $|\cdot|$ to denote the $\ell^1$-norm 
of $\R^d$, namely,
\begin{align*}
    \label{eq:}
    |x|=\sum_{j=1}^d|x_j|, \quad \text{for } x\in\R^d.
\end{align*}
The set of positive Radon measures on $\sd$ is denoted by $\M_+(\sd)$ and the set 
of signed bounded Radon measures on $\sd$ is denoted by $\M_b(\sd)$. Then
$\M_{+,b}(\sd):=\M_+(\sd)\cap\M_b(\sd)$ is the set of bounded positive Radon measures 
on $\sd$. The bounded Radon measures $\M_b(\sd)$ is a Banach space with respect to the total variation norm which we denote by $\|\cdot\|$.

The set of continuous functions from a topological space $X$ to a topological space $Y$ is denoted by $C(X;Y)$. The shorthand notation $C(X):=C(X,\R)$ will be used. The support of a given function $\varphi:X\to\R$ is denoted by $\supp\varphi$. The set of functions in $C(X)$ with compact support is denoted by $C_c(X)$. Note that the bounded functions in $C(X)$ is a Banach space with respect to the supremum norm $\|\varphi\|_\infty=\sup_{x\in X}|\varphi(x)|$. The closure of $C_c(X)$ in the Banach space of bounded and continuous functions is denoted by $C_0(X)$. The topology used in both $C_c(X)$ and $C_0(X)$ is the topology induced by the supremum norm $\|\cdot\|_\infty$. The set of functions $\varphi\colon\sd\times[0,\infty)\to\R$ such that for each fixed $x\in\sd$ the function $\varphi(x,\cdot)$ is continuously differentiable and $\varphi(\cdot,t)\in C_c(\sd)$ for all $t\in[0,\infty)$ is denoted by $C^1([0,\infty);C_c(\sd))$.

If $\omega\in C(\sd;[0,\infty))$ and $\mu\in\M_+(\sd)$, we define
the weighted measure $\omega\mu\in\M_+(X)$ to be the unique positive Radon measure
obtained from the Riesz Representation Theorem
with the functional mapping each compactly supported test function $\varphi\in
C_c(\sd)$ into $\int_X\varphi(x)\omega(x)\mu(dx)$. Moreover, if
$f:[0,\infty)\to\M_+(\sd)$, then $\omega f:[0,\infty)\to\M_+(\sd)$ is defined
to map each $t\in[0,\infty)$ into the weighted measure $\omega f(\cdot,t)$.

Note that due to the Riesz Representation Theorem $\M_+(\sd)$ is the dual of $C_c(\sd)$, and similarly by the Riesz-Markov-Kakutani Representation Theorem $\M_b(\sd)$ is the dual of $C_0(\sd)$. In both cases we use the dual pairing notation $\langle \varphi, \mu\rangle=\int_\sd \varphi(x)\mu(dx)$. Due to these Theorems we will some times treat measures as functionals and functionals as measures, e.g., use the same letter $\mu$ for the measure and for its corresponding functional $\varphi\mapsto \langle \varphi,\mu\rangle$.

In the set of signed bounded Radon measures $\M_b(\sd)$ we will some times use the topology induced by the total variation norm $\|\cdot\|$, and some times the weak*-topology from the dual of $C_0(\sd)$, i.e., the weakest topology in $\M_b(\sd)$ which makes $\mu\mapsto\langle \phi,\mu\rangle$ continuous for each $\phi\in C_0(\sd)$. We will always specify which topology is used.

For $I=[0,\infty)$ or $I=[0,T]$ for some $T>0$ we define $C^1(I;\M_b(\sd))$ as the set containing all those $f\in C(I;\M_b(\sd))$ for which there exists $\dot{f}\in C(I;\M_b(\sd)$ such that 
\begin{align*}
    \lim_{h\to0}\|h^{-1}\dot{f}(t)-(f(t+h)-f(t))\|=0
\end{align*}
for all $t$ in the interior of $I$, where $\M_b(\sd)$ is endowed with the topology induced by the total variation norm $\|\cdot\|$. The measure-valued function $\dot{f}$ is called the time derivative of $f$.

In subsection \ref{sec:stepIII} uniform spaces will be needed and necessary notations and definitions will be discussed therein. 

\subsection{Main results}
We list here the main results of the paper, namely, the existence Theorem \ref{thm:existence},
mass-conservation Theorem \ref{thm:mass_conservation}, and the gelation Theorem \ref{thm:gel}. 

\begin{theorem}
    \label{thm:existence}
    \textup{(Existence)}
    Suppose that $K$ is as in \eqref{eq:condK_cont} and
    satisfies  the upper bound 
    \eqref{eq:condK_sym2} with $-\beta\in\R$, and
    $\gamma,\lambda\in \R^2$ satisfying $-\lambda_j\leq\gamma_j+\lambda_j$ and  $\gamma_j+\lambda_j<1$ for $j=1,2$. Suppose
    that  a given initial data $f_0 \in \mathcal{M}_+(\sd)$ satisfies
    \eqref{eq:ini_cond}
  for some $r>0$.
    Then, there exists a weak
    solution $f:[0,\infty)\to\M_+(\sd)$ to
    \eqref{B4} in the sense of definition \ref{def:time-dep-sol-thesis} with
    $f(0,\cdot)=f_{0}$.
\end{theorem}

The above Theorem gives existence for weak solutions. We will also show that the following Corollary \eqref{cor:discrete_existence} holds, which gives the existence of strong solutions for the discrete coagulation equation \eqref{eq:B4_discrete}.

\begin{corollary}
    \label{cor:discrete_existence}
    \textup{(Existence discrete)} Suppose that $K$ is as in Theorem \ref{thm:existence}. Suppose that a given (discrete) initial data $n_0:\dsd\to[0,\infty)$ satisfies 
     \begin{equation}
       \label{eq:discrete_ini_cond}%
       \sum_{\alpha>0}|\alpha|n_0(\alpha)<\infty.
    \end{equation}
    Then there exists a strong solution $n:\dsd\times[0,\infty)\to[0,\infty)$ to the discrete 
    coagulation equation  \eqref{eq:B4_discrete} with $n(\cdot,t)=n_0$, namely, $t\mapsto n(\alpha,0)$ is continuously differentiable for all $\alpha\in\dsd$ and $n(\alpha,t)$ satisfies \eqref{eq:B4_discrete} for all $t\ge0$ and $\alpha\in\dsd$. Moreover, 
    \begin{align}
       \sum_{\alpha>0}
       \alpha n(\alpha,t)
       \leq
       \sum_{\alpha>0}
       \alpha n_0(\alpha),
       \quad 
       \text{and }
       \quad 
       \sup_{s\in[0,t]}
       \sum_{\alpha>0}
       \omega(\alpha) n(\alpha,s)
       <\infty 
    \end{align}
    for all $t\ge0$, where $\omega$ is  defined in \eqref{eq:weight}.
\end{corollary}

\begin{theorem}
    \label{thm:mass_conservation}
    \textup{(Conservation of mass)}
    Suppose that $K$ is as in \eqref{eq:condK_cont} and satisfies the upper
    bound \eqref{eq:condK_sym2} with $-\beta\in\R$, and
    $\gamma,\lambda\in \R^2$ satisfying $-\lambda_j\leq\gamma_j+\lambda_j$ and  $\gamma_j+\lambda_j\leq1$ for $j=1,2$. Suppose
    that $f_0\in\M_+(\sd)$ satisfies the moment condition \eqref{CondInVal} and that
    $f:[0,\infty)\to\M_+(\sd)$ is any weak solution to \eqref{B4} with the
    initial data $f_0$ in the sense of definition
    \ref{def:time-dep-sol-thesis}. If $\gamma_2\leq 1$,
    then 
    \begin{align}
        \label{eq:mass_conservation}
        \int_{\sd}xf(dx,t)=\int_{\sd}xf_0(dx) 
    \end{align}
    for all $t\ge0$.
\end{theorem}

Note that both the diffusion \eqref{eq:diffusion} and ballistic \eqref{eq:ballistic} kernels 
as well as their linear combinations satisfy the conditions of the Theorems \ref{thm:existence}
and \ref{thm:mass_conservation}. 

\begin{theorem}
    \label{thm:gel}
    \textup{(Gelation)}
    Suppose that $K$ is as in \eqref{eq:condK_cont} and satisfies the upper
    bound \eqref{eq:condK_sym2} with $-\beta\in\R$, and
    $\gamma,\lambda\in \R^2$ satisfying $-\lambda_j\leq\gamma_j+\lambda_j$ and  $\gamma_j+\lambda_j\leq1$ for $j=1,2$, and the lower bound \ref{eq:lower} with $\ggel\in(1,2)$ and $-\pgel\leq \ggel+\pgel$. Suppose
    that $f_0\in\M_+(\sd)$ satisfies the moment condition \eqref{CondInVal}  and that
    $f:[0,\infty)\to\M_+(\sd)$ is a weak solution to \eqref{B4} with the
    initial data $f_0$ in the sense of definition
    \ref{def:time-dep-sol-thesis}. If $\int_\sd |x|f_0(dx)\ne 0$,  then gelation occurs in finite time, namely, there exists $T_*\in(0,\infty)$ such that 
    \begin{align*}
        \int_\sd |x|f(dx,t)<\int_\sd |x|f_0(dx) \quad \forall t\ge T_*.
    \end{align*}
\end{theorem}

As mentioned in subsection \ref{sec:aim}, the proof of Theorem \eqref{thm:gel} follows from results in \cite{Dust} for $d=1$ after small changes required for the multicomponent case. 

\begin{remark}
    Note that Theorems \ref{thm:mass_conservation} and \ref{thm:gel} apply to any solution of the coagulation equation \eqref{B4} that defines a weak solution in the sense of Definition \ref{def:time-dep-sol-thesis}, including solutions to the continuous \eqref{eq:B4_continuous} and discrete \eqref{eq:B4_discrete} equations. For example, if $g : [0,\infty)\to L^1(\sd)$ is a solution to the continuous coagulation equation \eqref{eq:B4_continuous}, then $f:[0,\infty)\to\M_+(\sd)$,  $f(dx,t):= g(x,t) dx$ satisfies the weak coagulation equation \eqref{eq:thesis_weak}. If also the other conditions of Definition \ref{def:time-dep-sol-thesis} hold for this $f$, then one can use the above mentioned Theorems. Similarly for the discrete equation \eqref{eq:B4_discrete}. 
\end{remark}

\begin{remark}
    Note that in the Theorems \ref{thm:mass_conservation} and \ref{thm:gel}, we allow $\gamma_j+\lambda_j=1$ while 
    in the existence Theorem \ref{thm:existence} we require $\gamma_j+\lambda_j<1$ due to technical reasons. 
\end{remark}

\subsection{Plan of the paper and key ideas for proofs}
\label{sec:plan}

The proof of the existence Theorem \ref{thm:existence} will be the most involved of our results. In order to prove it we will proceed according to
the following four steps.
\begin{enumerate}[wide=0pt, leftmargin=\parindent, label=Step \arabic*]
    \item
        \label{step:1}
        For each $\epsilon\in(0,1)$ formulate a regularized coagulation
        equation which is more manageable.
    \item
        \label{step:2}
        Prove existence of unique regularized solution $f_\epsilon$ from each
        regularized equation labeled by $\epsilon\in(0,1)$.
    \item 
        \label{step:3}
        Prove that for the weight function $\omega$ defined in
        \eqref{eq:weight}, the closure of the weighted family $\{\omega
        f_\epsilon\}_{\epsilon\in(0,1)}$ of regularized solutions is
        sequentially compact in the topology of uniform convergence on compact
        sets.  Hence, there exist $f$ and a subsequence $(f_{\epsilon_n})$ of
        regularized solutions, such that $\omega f_{\epsilon_n}\to \omega f$ as
        $\epsilon_n\to0$.

    \item 
        \label{step:4}
        Prove that the limiting function $f$ satisfies the coagulation equation
        \eqref{B4} in the sense of definition \ref{def:time-dep-sol-thesis}.
\end{enumerate}

In section \ref{sec:regulized} we will do the \ref{step:1} and \ref{step:2}.
The proof of the existence Theorem \ref{thm:existence} is finished in section
\ref{sec:existence} where we will do the \ref{step:3} and \ref{step:4} and then combine them in subsection \ref{sec:existence_proof}. In the 
subsection \ref{sec:discrete} the Corollary \ref{cor:discrete_existence} is proved.

The
mass-conservation Theorem \ref{thm:mass_conservation} is proved in section
\ref{sec:mass}. The proof is essentially applying the Dominated Convergence Theorem to a specific sequence of test functions that are uniformly Lipschitz.
The gelation Theorem \ref{thm:gel} is proved in section \ref{sec:gel} and follows the argument from \cite[Theorem 2.4]{Dust} for $L^1$-valued solutions.  In the following section \ref{sec:def} we will give the precise
Definition \ref{def:time-dep-sol-thesis} of a weak solution to the coagulation equation \eqref{B4}.

%auto-ignore
% final sec 2
\section{Definitions and auxiliary results}
\label{sec:def}

We now give the definition of a weak solution to \eqref{B4}.
\begin{definition}                                                              
    \label{def:time-dep-sol-thesis}                                             

    Let $K$ be as in \eqref{eq:condK_cont} and satisfy the
    upper bound \eqref{eq:condK_sym2} for $\gamma,\lambda \in\R^2$ satisfying
    $-\lambda_j \leq\gamma_j+\lambda_j $ for $j=1,2$.  Let a continuous weight function
    $\omega:\sd\to(0,\infty)$ be defined by
    \begin{align}
        \label{eq:weight}
        \omega(x) =
        \begin{cases}
            |x|^{\min(-\beta,-\lambda_1 )} & \text{if } |x|\leq1,\\
            |x|^{\max(\gamma+\lambda )} & \text{if } |x|>1.
        \end{cases}
    \end{align}
    Suppose that                                                                                                                                     
    $f_{0}\in\M_{+}(\sd)$ satisfies
    \begin{equation}                                                            
        \label{CondInVal}
        \int_{\sd}|x|f_{0}(dx)                                 
        +\int_{\sd}\omega(x)f_0(dx)
        <\infty. 
    \end{equation}                                                              
    Let $\M_{b}(\sd)$ be endowed with the weak*-topology.
    A measure-valued function \\$f:[0,\infty)\to\M_{+}(\sd)$   
    is called a weak solution to \eqref{B4} with the initial data 
    $f_{0}$ if the following conditions are satisfied: 
    \begin{enumerate}[label=(\alph*)]
        \item The measure-valued function $f$ agrees with the initial data,
            \begin{align}
                \label{eq:agrees_initial}
                f(\cdot,0)=f_0.
            \end{align}
        \item
            The measure-valued function $f$ is continuous in time in the sense 
            that %\sos{This does not imply (c) because Mb is only pointwise bounded in the total variation norm}
            \begin{align}
                \label{eq:continuity}
                \omega f\in C([0,\infty);\M_b(\sd)).
            \end{align}
        \item 
            For all $t\ge0$ there holds
            \begin{equation}                                                            
                \label{eq:bound_gamma_moment}%                          
                \sup_{s\in\left[  0,t\right]  }                                         
                \int_{\sd}\omega(x)f(dx,s)       
                <\infty. 
             \end{equation}                                                              
        \item 
            The mass is not increasing, namely,
             \begin{equation}                                                            
                \label{eq:conservation}                       
                \int_{\sd}xf(dx,t)
                \leq
                \int_{\sd} xf_{0}(dx)
            \end{equation}                                                              
            for all $t\ge0$.
        \item 
            There holds
            \begin{align}                                                               
                &\int_{\sd}\varphi(x,t)f(dx,t)-\int_{\sd}\varphi(x,0)f_0(dx)          
                \nonumber                                                               
                =\int_0^t \int_{\sd}f(dx,s)\partial_s\varphi(x,s)ds                    
                \\                                                                      
                \label{eq:thesis_weak}                                                  
                &+\frac{1}{2}\int_0^t\int_{\sd}\int_{\sd}                             
                K(x,y)f(dx,s)f(dy,s)                                                    
                \left[\varphi(x+y,s)-\varphi(x,s)-\varphi(y,s)\right]ds            
            \end{align}                                                                 
            for all $t\ge0$ and for all test functions $\varphi\in
            C^1([0,\infty);C_c(\sd))$.
    \end{enumerate}
\end{definition}

Note that the time continuity condition \eqref{eq:continuity} of
$f$ implies for all $\varphi\in C_c(x)$ the continuity of $t\mapsto
\int_{\sd}\varphi(x)f(dx,t)$, since $\varphi/\omega\in C_c(x)$.

The weak formulation \eqref{eq:thesis_weak} is obtained from 
\eqref{eq:B4_continuous} by integrating over $x$ and integrating in time from $0$
to $t$.
To make sure that the conditions \eqref{CondInVal} and \eqref{eq:conservation}
guarantee that everything is well-defined in \eqref{eq:thesis_weak}, we
start by proving the following Proposition \ref{prop:weak_sol_well_defined}.

\begin{proposition}
    \label{prop:weak_sol_well_defined}
    Let $K$ be as in \eqref{eq:condK_cont} and satisfy the upper bound
    \eqref{eq:condK_sym2} with $-\lambda \leq\gamma+\lambda $. Let $\omega$ be as in
    \eqref{eq:weight} and let $f: [0,\infty)\to\M_+(\sd)$. Then
    \begin{align}
        \label{eq:weak_sol_coag_bound}
        \int_{\sd}\int_{\sd} K(x,y)f(dx,t)f(dy,t) \leq
        2c_2
        \left(
            \int_{\sd}\omega(x)f(dx,t) 
        \right)^2
    \end{align}
    for all $t\ge0$, where $c_2>0$ is the constant from the upper bound
    \eqref{eq:condK_sym2}.
\end{proposition}

\begin{proof}
    This follows from $K(x,y)\leq2\omega(x)\omega(y)$.
\end{proof}

Together with the condition \eqref{eq:bound_gamma_moment}, the above Proposition 
\ref{prop:weak_sol_well_defined} imply uniform boundedness for the
time-integrand of the coagulation term in the right hand side of
\eqref{eq:thesis_weak}, i.e.
\begin{align*}
    &
    \sup_{s\in[0,t]}
    \left| 
        \int_{\sd}\int_{\sd}                             
        K(x,y)f(dx,s)f(dy,s)                                                    
        \left[\varphi(x+y,s)-\varphi(x,s)-\varphi(y,s)\right]ds     
    \right| 
    \\
    &\leq 
    6c_2
    \sup_{s\in[0,t]}
    \|\varphi(\cdot,s)\|_\infty
    \left( 
        \sup_{s\in\left[  0,t\right]  }                                         
            \int_{\sd}\omega(x)f(dx,s)       
    \right) ^2
    \\
    &<\infty.
\end{align*}
The uniform bound above, together with the
continuity of $f$ in time by condition
\eqref{eq:continuity} guarantees the
well-definedness of the coagulation term in \eqref{eq:thesis_weak}. The other
terms in \eqref{eq:thesis_weak} are well-defined by continuity of $f$,
continuous differentiability in time of the test function $\varphi$ and its
compact support in size. Note that because of continuity of the time-integrands on the left hand side of \eqref{eq:thesis_weak}, we have that $t\mapsto \int_\sd \varphi(x,t)f(dx,t)$ is continuously differentiable. 

\begin{proposition}
    \label{prop:valid_test_functions}
    \textup{(Valid test functions)}
    Let $\phi:\sd\to[0,\infty)$ be defined as the pointwise limit of some sequence $(\varphi_n)\subset C_c(\sd)$. Suppose there exists $c,C\in(0,\infty)$ such that $\phi(x)\leq c|x|$ and $\phi(x)\leq C$ for all $x\in\sd$. Then any solution $f$ to \eqref{B4} in the sense of Definition \ref{def:time-dep-sol-thesis} satisfies the coagulation equation \eqref{eq:thesis_weak} with the test function $\varphi(x,t)=\phi(x)$.
\end{proposition}
\begin{proof}
    The Dominated Convergence Theorem together with the Proposition \ref{prop:weak_sol_well_defined} and the fact that $\varphi_n$ is uniformly bounded due to $\phi\leq C$ imply
    \begin{align*}
        &\lim_{n\to \infty}
        \int_{\sd}
        \int_{\sd}
        f(dx,t)f(dy,t)K(x,y)[\varphi_n(x+y)-\varphi_n(x)-\varphi_n(y)]
        \\
        =&
        \int_{\sd}
        \int_{\sd}
        f(dx,t)f(dy,t)K(x,y)[\phi_n(x+y)-\phi_n(x)-\phi_n(y)].
    \end{align*}
    Similarly, the Dominated Convergence Theorem together with finiteness of the mass \eqref{CondInVal}, \eqref{eq:conservation} and $\phi(x)/|x|\leq c$ imply 
    \begin{align*}
        \lim_{n\to \infty}
        \int_{\sd}
        \varphi_n(x)f(dx,t)
        =&
        \lim_{n\to \infty}
        \int_{\sd}
        \frac{\varphi_n(x)}{|x|}|x|f(dx,t)
        =
        \int_{\sd}
        \phi(x)f(dx,t)
    \end{align*}
    for all $t\ge0$. This proves the claim.
\end{proof}

%auto-ignore
% final sec 3
\section{The regularized problem}
\label{sec:regulized}

\subsection{Step 1: regularization}

In this subsection we will do the \ref{step:1}. For each $\epsilon\in(0,1)$ we will regularize the dynamics to the compact subset 
$\{\epsilon\leq|x|\leq2/\epsilon\}$ of $\sd$. This is done according to the following
definition.

\begin{assumption}
    \label{ass:Astep2}
    Suppose we are given 
    \begin{enumerate}[label=(\alph*)]
        \item  a regularization parameter $\epsilon\in(0,1)$,
        \item
            \label{ass:Astep2_K}
            a continuous, non-negative and symmetric function $K:
            \sd\times\sd\rightarrow {\R}_+$,
        \item
            \label{ass:Astep2_zeta}
            a continuous function $\zeta_\epsilon:\sd\to [0,1]$, such that
            $\zeta _\epsilon\left( x\right) =1$ for $|x|\leq 1/\epsilon$,
            and $\zeta _\epsilon \left( x\right) =0$ for $|x|\geq 2/\epsilon$.
    \end{enumerate}
\end{assumption}

In the regularized problem we do not need anymore assumption on the kernel $K$,
than what is in Assumption \ref{ass:Astep2}\ref{ass:Astep2_K}. The reason for this is that we are 
working in the compact set $\{\epsilon\leq|x|\leq2/\epsilon\}$ instead of the whole 
space $\sd$.

\begin{definition}
    \label{def:requ_solu}
    Suppose that Assumption \ref{ass:Astep2} holds. Suppose our initial
    data $\feo\in\M_+(\sd)$ satisfies 
    \begin{align}
        \label{eq:reguCondInVal}
        \feo
        \left(
            \sd\setminus\{\epsilon\leq|x|\leq2/\epsilon\}
        \right)
        =0.
    \end{align}
    Let our time domain be either $I=[0,\infty)$ or $I=[0,T]$ for some $T>0$.
    Let $\M_{+,b}(\sd)$ be endowed with the total variation norm topology of
    the Banach space $\M_b(\sd)$. We say that $\fe\in
    C^1(I;\M_{b}(\sd))\cap C(I;\M_{+,b}(\sd))$ is a regularized
    solution with the initial condition $\feo$, if for every $t\in I$ the
    following conditions are satisfied:
    \begin{enumerate}[label=(\alph*)]
        \item 
            \begin{align}
                \label{eq:regu_t0}
                \fe(\cdot,0) &= \feo.
            \end{align}
        \item 
            \begin{align}
                \label{fVanish}
                \fe 
                \left(
                        \sd\setminus\{\epsilon\leq|x|\leq2/\epsilon\}
                        ,t
                \right) = 0.
            \end{align}
        \item 
            \begin{equation}
            \int_{\sd}\fe(dx,t)\leq  \int_{\sd}\feo(dx). \label{fEstim}
            \end{equation}%
        \item 
            for all $\varphi\in C^1(I;C_c(\sd))$.
            \label{eq:regu_measure}
            \begin{align}
                \nonumber
                &\int_{\sd}\varphi(x,t)\dot{\fe}(dx,t)  
                \\
                \label{coagEq_regu_measure2} 
                &=
                \frac{1}{2}\int_{\sd}\int_{\sd} 
                K(x,y)
                \left[\zeta_\epsilon(x+y)\varphi(x+y,t)-\varphi(x,t)-\varphi(y,t)\right]
                \fe(dx,t)\fe(dy,t).
            \end{align}
    \end{enumerate}
\end{definition}

\begin{remark}
    \label{re:regu_sol}
    The condition \ref{eq:regu_measure} is
    equivalent to its integrated version
    \begin{align}
         \nonumber  
        &\int_{\sd}\varphi(x,t)\fe(dx,t)
        -
        \int_{\sd}\varphi(x,0)\feo(dx)
        =
        \int_0^t\int_{\sd}f(dx,s)\partial_s\varphi(x,s)ds
        \\
         \label{coagEq_regu_measure1} 
        &+\frac{1}{2}\int_0^t\int_{\sd}\int_{\sd} 
        K(x,y)
        \left[\zeta_\epsilon(x+y)\varphi(x+y,s)-\varphi(x,s)-\varphi(y,s)\right]
        \fe(dx,s)\fe(dy,s)ds, 
    \end{align}
    holding for all $\varphi\in C^1(I;C_c(\sd))$.  Also note that we can
    choose time-independent test functions in both \eqref{coagEq_regu_measure2}
    and \eqref{coagEq_regu_measure1}. Moreover, by the condition \eqref{fVanish},
    the equations \eqref{coagEq_regu_measure2} and \eqref{coagEq_regu_measure1}
    hold for all test functions $\phi\in C^1(I;C_0(\sd))$ if they hold for
    all $\varphi\in C^1(I;C_c(\sd))$.
\end{remark}

\subsection{Step 2: solving the regularized problem }
We will do the \ref{step:2}, that is, we will uniquely solve the regularized coagulation equation
by proving the following Theorem.

\begin{theorem}
    \label{thm:regularized_existence_uniqueness}
    Suppose that Assumption \ref{ass:Astep2} holds. Then, for any initial
    condition $\feo\in \mathcal{M}_+(\sd)$ satisfying \eqref{eq:reguCondInVal},
    there exists a unique regularized solution $\fe\in
    C^1([0,\infty);\M_{b}(\sd))\cap C([0,\infty);\M_{+,b}(\sd))$ with the
    initial condition $\feo$ in the sense of definition \ref{def:requ_solu}.
\end{theorem}

In order to prove the above Theorem \ref{thm:regularized_existence_uniqueness}, we first prove the following Proposition \ref{prop:unique_weak_sol_T} that gives existence and uniqueness for short times.

\begin{proposition}
    \label{prop:unique_weak_sol_T}
    Suppose that Assumption \ref{ass:Astep2} holds and $\feo\in\M_+(\sd)$
    satisfies the condition \eqref{eq:reguCondInVal}.  Let $T\in(0,\infty)$
    satisfy $T\leq\frac{1}{12\|K\|_\epsilon(1+\|\feo\|)^2}$, where $\|K\|_\epsilon:=\sup\{|K(x,y)|\ : \ |x|,|y|\in[\epsilon,2/\epsilon]\}$.  Then there exists
    a unique regularized solution $\fe\in C^1([0,2T];\M_{b}(\sd))\cap
    C([0,2T];\M_{+,b}(\sd))$ in the sense of definition \ref{def:requ_solu}.
\end{proposition}

\begin{proof}
    This proof is a standard Banach fixed point argument. 
    For example, 
    the proof of \cite[Proposition 3.6]{FLNV} is essentially the same. Also a
    very detailed proof of our claim, Proposition \ref{prop:unique_weak_sol_T}, in
    the case $d=1$ can be found in \cite{SakariThesis}.
    Therefore, we omit some standard computations.
    
    First step is to reformulate the question as a fixed point equation $\op[f]=f$.
     We start by collecting all positive
    Radon measures satisfying the condition \eqref{eq:reguCondInVal} into a set
    $\X$. Then $\X\subset \M_{+,b}(\sd)$.
    By the condition \eqref{fVanish}, we are looking for solutions in $C([0,T];\X)$. We equip the space
    $C([0,T];\X)$ with the norm
    $\|f\|_T:=\sup_{t\in[0,T]} \|f(\cdot,t)\|$,
    where $\|\cdot\|$ is the total variation norm.
    We define our fixed point operator $\op:C([0,T];\X)\to C([0,T];\X)$ by 
    setting 
    \begin{align}
    \nonumber
        &\langle \phi, \op[g](t)\rangle
        := \int_{\sd}\phi(x) e^{-\int_0^t a[g](x,s)ds}\feo(dx),
        \\
        \label{eq:fixed_point_op}
        &\quad + \frac{1}{2}\int_0^t \int_{\sd} \int_{\sd}
        K(x,y)\zeta_\epsilon(x+y)\phi(x+y)
        e^{-\int_s^t a[g](x+y,\xi)d\xi}g(dx,s)g(dy,s)ds
    \end{align}
    for all $t\in[0,T]$ and $\phi\in C_0(\sd)$. Where we defined $a[g]\in C(\sd\times [0,T];[0,\infty))$ by
    \begin{align}
      \label{aFunct}
    a\left[ g\right] \left( x,t\right)
    &:=
    \int_{{\sd}}K\left(
    x,y\right) {g\left( dy,t\right)}.
    \end{align}
    
    Motivation for the definition of $\op$, is that now if $\fe$ is a regularized solution with the initial
    condition $\feo$ in the sense of definition \ref{def:requ_solu}, then
    $\op[\fe]=\fe$. To see this, fix $t\in[0,T]$ and
    $\phi\in C_0(\sd)$. Let $\varphi\in C^1([0,T];C_0(\sd))$ be 
    defined by  $\varphi(x,s)=\phi(x)e^{-\int_s^t
    a[\fe](x,\xi)d\xi}$. Then $\langle \phi, \op[\fe](t)\rangle = \langle\phi, \fe(\cdot,t)\rangle$ follows from a straightforward computation after putting $\varphi$ as the test function to the regularized coagulation equation \eqref{coagEq_regu_measure2}.

    Let $X=\{g\in C([0,T];\X)\ : \ \|g-\feo\|_T\leq 1$. Then $X$ is a complete metric space as a closed subset  of the Banach space $C([0,T];\M_b(\sd))$. Next we do a standard Banach fixed point argument to prove that $\op$ has a unique fixed point on $X$. 
    To prove that $\op$ is contraction, 
     we
    estimate the two terms $\opi$
    and $\opii$ of $\op$ separately. Here by $\opi$ we mean the term in \eqref{eq:fixed_point_op} with the integral over $\feo$  and $\opii:=\op-\opi$. Let 
    $\phi\in C_0(\sd)$ with $\|\phi\|_\infty\leq1$ and $t\in[0,T]$. 
    By 
    the fact that $|e^{-x}-e^{-y}|\leq |x-y|$ for $x,y\ge0$,
    we have 
    \begin{align}
        |\langle \phi, \opi[f](t)-\opi[g](t)
        \rangle|
        \label{eq:T1_contract_est}
        \leq &
        T
        \|K\|_\epsilon
        \|\feo\|
        \|f-g\|_T.
    \end{align}
    Similarly the second term can be estimated by 
    \begin{align}
        \label{eq:T2_contract_est}
        |\langle \phi, \opii[f](t)-\opii[g](t)
        \rangle|
        \leq 
        \frac{1}{2} T^2
        \|K\|_\epsilon^2
        \|f\|_T^2
        \|f-g\|_T 
        +
        T
        \|K\|_\epsilon
        \|f\|_T
        \|f-g\|_T.
    \end{align}
    Note that $\|h\|_T\leq
    1+\|\feo\|$ for all $h\in X$.
    Then \eqref{eq:T1_contract_est}, \eqref{eq:T2_contract_est} and
    $T\leq\frac{1}{12\|K\|_\epsilon(1+\|\feo\|)^2}$ imply
    \begin{align}
        \label{eq:contraction2}
        \|\op[f]-\op[g]\|_T
        \leq \frac{1}{2}\|f-g\|_T
    \end{align}
    for all $f,g\in X$. This proves that $\op$ is a contraction on $X$.

    Let $g\in X$. To prove that $\op[g]\in X$, we  
    telescope with $\op[\feo]$, use \eqref{eq:contraction2} and $\op=\opi+\opii$ to obtain
    \begin{align*}
        \|\op[g]-\feo\|_T \leq&
        \frac{1}{2}
        +
         \|\opi[\feo]-\feo\|_T +
         \|\opii[\feo]\|_T. 
    \end{align*}

     Take $t\in[0,T]$ and $\phi\in C_0(\sd)$ with $\|\phi\|_\infty\leq1$.
    Using $|e^{-x}-1|\leq x$ for
    $x\ge0$ implies
    \begin{align}
        |\langle \phi, \opi[\feo](t)-\feo
        \rangle|
        \leq &
         T
        \|K\|_\epsilon
        \|\feo\|^2
        %\\
        \leq
        \label{eq:T1_mapsto_itself}
        \frac{1}{12}.
    \end{align}
    Since $|\zeta_\epsilon(x+y)\phi(x+y) e^{-\int_s^t
    a[\feo](x+y,\xi)d\xi}|\leq 1$, the second term can be estimated by
    \begin{align}
        |\langle \phi, \opii[\feo](t)
        \rangle|
        \leq &
        \frac{1}{2}
        T
        \|K\|_\epsilon
        \|\feo\|^2
        \leq
        \frac{1}{24}.
        \label{eq:T2_mapsto_itself}
    \end{align}
    Hence, $\|\op[g]-\feo\|_T \leq 1/2+1/12+1/24\leq1$, and so, $\op[g]\in X$.
    Thereby, $\op$ maps $X$ to itself. Thus, 
    Banach fixed point Theorem gives the existence of a unique $\fe \in X$ satisfying $\op[\fe]=\fe$.
    It remains to be proven that $\fe$ is a unique solution to the regularized coagulation equation in the sense of definition \ref{def:requ_solu}.

    To prove uniqueness, suppose that $g_\epsilon\in
    C^1([0,T];\M_{+,b}(\sd))$ is a regularized solution with the initial
    data $\feo$ in the sense of definition \ref{def:requ_solu}.  Then $\op[g_\epsilon]=g_\epsilon$ by the argument after equation \eqref{aFunct}. Putting
    $\phi\in C_0(\sd)$ as a test function to the regularized coagulation equation
    \eqref{coagEq_regu_measure1} and using  $T\leq\frac{1}{12\|K\|_\epsilon(1+\|\feo\|)^2}$
    together with the condition \eqref{fEstim} imply
        $\left| \langle \phi,g_\epsilon(\cdot,t)-\feo\rangle \right| \leq 1$. Hence, $g_\epsilon\in X$. On the other hand, $\op$ has a unique fixed point on $X$. Therefore, $g_\epsilon=\fe$. This proves uniqueness. 

    It remains to be proven that $\fe$ is differentiable and satisfies the conditions \eqref{eq:regu_t0}-\eqref{coagEq_regu_measure1} of Definition \ref{def:requ_solu}. The condition \eqref{fVanish} holds as $\fe\in C([0,T];\X)$. The condition \eqref{eq:regu_t0} is true, since $\fe=\op[\fe]$ and $\op[\fe](0)=\feo$ by definition \eqref{eq:fixed_point_op} of $\op$.

    We define $g\in C([0,T];\M_b(\sd))$ by setting
    \begin{align}
        \label{eq:derivative}
        \langle \phi, g(t)\rangle :=
                \frac{1}{2}\int_{\sd}\int_{\sd}
                K(x,y)
                \left[\zeta_\epsilon(x+y)\phi(x+y)-\phi(x)-\phi(y)\right]
                \fe(dx,t)\fe(dy,t),
    \end{align}
    for $\phi\in C_0(\sd)$ and $t\in[0,T]$.
    Note that $g(t)$ is the functional on the right hand side of the regularized
    equation \eqref{coagEq_regu_measure2}.  
      Hence, proving $\dot{\fe}(\cdot,t)=g(t)$ for
    every $t\in[0,T]$ will prove differentiability of $\fe$ and that \eqref{coagEq_regu_measure2} holds.
    By continuity of $g$ it is sufficient to prove $\dot{\fe}(\cdot,t)=g(t)$ for
    $t\in(0,T)$.
    Let $t\in(0,T)$.
    Recall that $\dot{\fe}(\cdot,t)=g_t$ means that
    \begin{align}
        \label{eq:limit_derivative}
        |h|^{-1}\|\fe(\cdot,t+h)-\fe(\cdot,t)-hg(t)\|\to0
    \end{align}
    as $h\to0$. The limit \eqref{eq:limit_derivative} follows from standard computations after writing $\fe=\op[\fe]$.
    For details see \cite{SakariThesis}.

    Take $\varphi \in C_c(\sd)$, such that $\varphi(x)=1$
    whenever
    $|x|\in[\epsilon,4/\epsilon]$. Then
    $\zeta_\epsilon(x+y)\varphi(x+y)-\varphi(x)-\varphi(y)\leq0$ whenever 
    $|x|,|y|\in[\epsilon,2/\epsilon]$. Hence, putting  $\varphi$ into the regularized
    coagulation equation
    \eqref{coagEq_regu_measure1} and using condition \eqref{fVanish} imply condition \eqref{fEstim}.
    Thus $\fe$ is a
    regularized solution in the sense of Definition \ref{def:requ_solu}.
\end{proof}

To prove the Theorem \ref{thm:regularized_existence_uniqueness} we just have to 
extend the solution obtained from Proposition \ref{prop:unique_weak_sol_T} to
the whole time interval $[0,\infty)$. This is done by gluing as follows.
\\

\begin{proofof}[Proof of Theorem \ref{thm:regularized_existence_uniqueness}]
    The idea of the proof is to use induction and Proposition
    \ref{prop:unique_weak_sol_T}. 
    Fix $T\in(0,\infty)$ to be
    such that it satisfies $2T\leq\frac{1}{12\|K\|_\epsilon(1+\|\feo\|)^2}$.
    To shorten the notation, let 
    \[
        \mathcal{C}^n=C^1([0,(2+n)T];\M_{b}(\sd))\cap C([0,(2+n)T];\M_{+,b}(\sd)),
    \]
    for $n=0,1,2,\dots$.
    Let $\fe^0$ be the unique
    regularized solution with the initial condition $\feo$ obtained from the
    Proposition \ref{prop:unique_weak_sol_T}. Then $\fe^0\in \mathcal{C}^0$.

    Let $n\in\N\cup\{0\}$. Suppose that $\fe^n\in \mathcal{C}^n$ is a unique
    regularized solution 
    in the sense of definition
    \ref{def:requ_solu} with $\fe^n(\cdot,0)=\feo$. Let
    $\tilde{\feo}^{n+1}=\fe^n(\cdot,(1+n)T)$. Then
    $\|\tilde{\feo}^{n+1}\|\leq\|\feo\|$ by \eqref{fEstim}. Hence, by
    Proposition \ref{prop:unique_weak_sol_T} there exists a unique solution
    $\tilde{\feo}^{n+1}\in \mathcal{C}^0$ in the sense
    of definition \ref{def:requ_solu}
    with $\tilde{\fe}^{n+1}(\cdot,0)=\tilde{\feo}^{n+1}$. Note that also
    $t\mapsto \fe^n(\cdot,t+(1+n)T)$ is a regularized solution on the interval
    $[0,2T]$ in the sense of definition \ref{def:requ_solu} with initial data
    $\fe^n(\cdot,(1+n)T)$. Hence, $\tilde{\feo}^{n+1}=\fe^n(\cdot,(1+n)T)$ and
    uniqueness by Proposition \ref{prop:unique_weak_sol_T} imply that there holds
    $\tilde{\fe}^{n+1}(\cdot,t)=\fe^n(\cdot,(1+n)T+t)$ for all $t\in[0,T]$.

    Let $\fe^{n+1}\in \mathcal{C}^{n+1}$ be
    defined by
    $\fe^{n+1}(\cdot,t)=\fe^{n}(\cdot,t)$ for all $t\in[0,(2+n)T]$ and
    $\fe^{n+1}(\cdot,t)=\tilde{\fe}(\cdot,t-(1+n)T)$ for all $t\in
    [(1+n)T,(3+n)T]$.  Then $\fe^{n+1}$ is a unique regularized solution in the
    sense of definition \ref{def:requ_solu} on the interval $[0,(3+n)T]$ with
    $\fe^{n+1}(\cdot,0)=\feo$. Hence, induction implies that for every
    $n\in\N\cup\{0\}$ there exists a unique regularized solution  $\fe^n\in
    \mathcal{C}^n$ in the sense
    of definition \ref{def:requ_solu} with $\fe^n(\cdot,0)=\feo$.

    Notice that by uniqueness, for every $n,N\in\N\cup\{0\}$ with $n\leq N$
    there holds $\fe^n(\cdot,t)=\fe^N(\cdot,t)$ for all $t\in[0,(2+n)T]$.
    Thereby, we can define $\fe\in C^1([0,\infty);\M_{b}(\sd))\cap
    C([0,\infty);\M_{+,b}(\sd))$ by $\fe(\cdot,t)=\fe^n(\cdot,t)$  for all
    $t\in[0,(2+n)T]$ and for all $n\in\N\cup\{0\}$. This makes $\fe$ into a
    unique regularized solution in the whole time interval $[0,\infty)$ in the
    sense of definition \ref{def:requ_solu} with $\fe(\cdot,0)=\fe$.
\end{proofof}

%auto-ignore
% final sec 4
\section{Proof of the existence result}
\label{sec:existence}

\subsection{Step 3: sequential compactness}
\label{sec:stepIII}

In this subsection we will do the \ref{step:3}. We start by collecting the
family of regularized solutions into one Assumption:

\begin{assumption}
    \label{ass:Astep3}

    Suppose $\M_{+,b}(\sd)\subset \M_b(\sd)$ is endowed with the weak*-topology. Suppose we are
    given:
    \begin{enumerate}[label=(\alph*)]
        \item
            \label{ass:Astep3_K}
            A kernel $K$ as in \eqref{eq:condK_cont} and it satisfies the
            upper bound \eqref{eq:condK_sym2}
            with $\gamma,\ \lambda \in\R^3$, such that $ -\lambda _l\leq\gamma_l+\lambda _l$ and
            $\gamma_l+\lambda _l<1$ for each $l\in\{1,2,3\}$.
        \item
            \label{ass:Astep3_weight}
            The continuous weight function $\omega:\sd\to(0,\infty)$ defined in
            \eqref{eq:weight}, that is, $\omega(x)=|x|^{\min\{-\lambda _l\}}$ for $|x|\leq 1$ and
            $\omega(x)=|x|^{\max\{\gamma_l+\lambda _l\}}$ for $|x|>1$.
        \item
            \label{ass:Astep3_ini}
            An initial data $f_0\in\M_+(\sd)$ and a constant $r>0$
            satisfying \eqref{eq:ini_cond}.
        \item
            \label{ass:Astep3_zeta}
            For each $\epsilon\in(0,1)$, a continuous function
            $\zeta_\epsilon:\sd\to [0,1]$ such that $\zeta _\epsilon(
            x) =1$ for $|x|\leq 1/\epsilon$, and $\zeta _\epsilon(x) =0$ for
            $|x|\geq 2/\epsilon$.
        \item
            \label{ass:Astep3_fam}
            Family of weighted regularized solutions 
            $F=\{\omega f_\epsilon\}_{\epsilon\in(0,1)}$, where $f_\epsilon$ is a regularized
            solution in the sense of Definition \ref{def:requ_solu} with the
            initial data $f_{\epsilon,0}(\cdot):=f_0(\cdot\cap
            \{\epsilon\leq|x|\leq2/\epsilon\})$, truncation $\zeta_\epsilon$, and the
            kernel $K$.
    \end{enumerate}
\end{assumption}

\begin{remark}
    {Note that $F\subset
            C([0,\infty);\M_{+,b}(\sd))$ with $\M_{+,b}(\sd)$ endowed
            with weak*-topology since by Definition 
            \ref{def:requ_solu} we know that $F\subset C([0,\infty);\M_{+,b}(\sd))$ with $\M_{+,b}(\sd)$ endowed with the topology induced by the
            total variation-norm.}
\end{remark}

\subsubsection*{Space for the convergence}

As explained in \ref{step:3}, we want to prove that the family $F$ of
regularized solutions is sequentially compact in the topology induced by the
uniformity of uniform convergence in compact sets. We remark this topology is equivalent
with the compact open topology \cite[Theorem 11 on p. 230]{Kelley}, which will be used below in the proof of Proposition \ref{prop:sequentally_compactness}. We suppose 
that the reader is familiar with the basic concepts and results of uniform spaces (cf.
\cite[chapters 6 and 7]{Kelley} or \cite{SakariThesis}). 

To obtain the
sequential compactness we will first prove that the closure of $F$ is compact and 
then prove its metrizability to obtain the sequentially compactness. The compactness 
will be obtained from Arzelà-Ascoli Theorem \cite[Theorem 17 on p.233-234]{Kelley}, 
which we present for the readers convenience below.

We are going to need also two simple results, Proposition
\ref{prop:handling_closure} and \ref{prop:sequentally_compactness}, related to
uniform spaces. We will give the proofs of these results, since we were 
unable to find them elsewhere. The first, Proposition \ref{prop:handling_closure}, 
will be needed in our case to control the closure of the regularized
solutions $F$. The second, Proposition \ref{prop:sequentally_compactness},
will be used to obtain the {pseudo-}metrizability of the closure of $F$. We will apply the Arzelà-Ascoli Theorem \ref{thm:Arzela-Ascoli} and Proposition \ref{prop:sequentally_compactness} to 
$X=[0,\infty)$ and $Y=\M_b(\sd)$ with the weak star topology. While Proposition \ref{prop:candidate_bounds} will be applied for 
$Z=C_0(\sd)$, and so $Z^*\cong \M_b(\sd)$.

\begin{theorem}
    \label{thm:Arzela-Ascoli}
    \textup{(Arzelà-Ascoli)} Let $C(X;Y)$ be the family of all continuous
    functions on a $T_3$ locally compact topological space $X$ to a Hausdorff
    uniform space $Y$, and let $C(X;Y)$ have the topology of uniform
    convergence on compact sets. Then a subfamily $F$ of $C(X;Y)$ is compact if and
    only if
    \begin{enumerate}[label=(\alph*)]
        \item \label{thm:AA_a} 
            $F$ is closed in $C(X;Y)$,
        \item \label{thm:AA_b}
            $F[x]:=\{f(x) \ | \ f\in F\}$ has a compact closure for all $x\in X$, and
        \item \label{thm:AA_c}
            the family $F$ is equicontinuous. 
    \end{enumerate}
\end{theorem}
In the above Theorem \ref{thm:Arzela-Ascoli}, by $X$  being $T_3$ we mean that every closed set $E\subset X$ and point in $X\setminus E$ have disjoint neigborhoods. We note that Kelley uses different naming conventions in \cite{Kelley}, e.g. our $T_3$ is called regular in \cite{Kelley}.

We recall that equicontinuity of $F$ in the condition \ref{thm:AA_c} means that for all $x\in X$ and for each member $V$ in the uniformity of $Y$ there is a neigbourhood $U\subset X$ of $x$, such that $f[U]\subset V[f(x)]$ for all $f\in F$. Here, $f[U]$ denotes the image of $U$ and $V[f(x)]=\{ y\in Y\ : \ (f(x),y)\in V$. The latter shorthand notation is very useful, and therefor if $X$ is a set and $U\subset X^2$, then we denote
\begin{align}
    U[x] :=\{y\in X:(x,y)\in U\}.
\end{align}
We will also need the following two shorthand notations. If $d:X^2\to [0,\infty)$ and $r>0$, then we denote 
\begin{align}
    \mathbb{B}_{d,r} := \{(x,y)\in X^2\ : \ d(x,y)<r\}. \label{eq:unicormity_V_set}
\end{align}
If $X$ is a vector space and $p:X\to[0,\infty)$, then we denote
\begin{align}
    \mathbb{B}_{p,r} := \{(x,y)\in X^2\ : \ p(x-y)<r\}. \label{eq:unicormity_V_set_norm}
\end{align}

\begin{proposition}
    \label{prop:handling_closure}
    Let $X$ be a topological space and $Z$ a locally compact Hausdorff space.
    Let $Z^*$ be equipped with the uniformity generated by the pseudo-norms
    obtained for each $z\in Z$ by the dual-pairing $z^*\mapsto|\langle
    z,z^*\rangle|$. Let $C(X;Z^*)$ be equipped with the uniformity of uniform
    convergence on compact sets. Let $F\subset C(X;Z^*)$ and denote its closure
    by $\bar{F}$. If $\bar{f}\in \bar{F}$, then for every $z\in Z$, $K\subset
    X$ compact and $\delta>0$, there exists $f\in F$, such that
    \begin{align}
        |\langle z,\bar{f}(x)-f(x)\rangle|<\delta,
    \end{align}
    for all $x\in K$.
\end{proposition}
\begin{proof}
    Fix any $\bar{f}\in\bar{F}$ and fix $z\in Z$, $K\subset X$ compact and
    $\delta>0$.
    By definition, $\mathbb{B}_{|\langle
    z,\cdot\rangle|,\delta}=\{(z^*_1,z^*_2)\in (Z^*)^2 \ : |\langle
    z,z^*_1-z^*_2\rangle|<\delta \}$ is in the uniformity of $Z^*$.  Since $K$
    is compact, we obtain that
    \begin{align*}
        U &:= \{(f,g)\in C(X;Z^*)^2 \ :
        \ (f(x),g(x))\in \mathbb{B}_{|\langle
        z,\cdot\rangle|,\delta} \quad \forall x\in K\} \\
        &=\{(f,g) \ : \ |\langle z,g(x)-f(x)\rangle| < \delta \quad \forall x\in K\}
    \end{align*}
    is in the uniformity of $C(X;Z^*)$. Then by \cite[Theorem 4 on p. 178]{Kelley},
    the set 
    \[
        U[\bar{f}]=\{f\in C(X;Z^*)\ : \ |\langle z,\bar{f}(x)-f(x)\rangle|<\delta \quad \forall x\in K\}
    \]
    contains a neighborhood of $\bar{f}$. Hence, by definition of closure,
    there exists $f\in U[\bar{f}]\cap F$.
\end{proof}
\begin{proposition}
    \label{prop:sequentally_compactness}
    Let $X$ be a $\sigma$-compact topological space.
    Let $Y$ be a uniform
    space, whose topology is pseudo-metrizable by a pseudo-metric $d$. Let $F\subset
    C(X;Y)$, where $C(X;Y)$ is equipped with the uniformity of uniform
    convergence on compact sets. Then the topology of $F$ 
    is pseudo-metrizable.
\end{proposition}
\begin{proof}
    By Theorem \cite[Theorem 11 on p.230]{Kelley}, while the uniformity of $C(X;Y)$ depends on the uniformity of $Y$, the topology of $C(X;Y)$ depends only on the topology induced by the uniformity of $Y$.
    Therefore, we can equip $Y$ with the uniformity induced by the pseudo-metric $d$ without changing the topology of $C(X;Y)$ nor $F$.

    Since $X$ is $\sigma$-compact, there exist compact sets $K_1\subset K_2\subset
    \dots\subset X$, such that $\bigcup_{n=1}^\infty K_n = X$.
    For each $n\in\mathbb{N}$
    we define
    \[
        U_n := \{(g,h) \in C(X;Y)^2 \ : \ (g(x),h(x)) \in \mathbb{B}_{d,1/n} \
        \forall x\in K_n \}.
    \]

    Since $\{\mathbb{B}_{d,1/n}\}_{n=1}^\infty$ is the basis of the uniformity 
    induced to $Y$ by the pseudo-metric $d$ and $\bigcup_{n=1}^\infty K_n = X$, we 
    obtain that $\{U_n\}_{n=1}^\infty$ is a basis for the uniformity of $C(X;Y)$.
    Hence, the pseudo-metrizability follows from the Metrization Theorem
    \cite[Theorem 13 on p. 186]{Kelley}.
\end{proof}

\subsubsection*{Uniform bounds}

We will now prove some bounds that are uniform in $\epsilon$ and $t$. These bounds 
are needed to prove the conditions \ref{thm:AA_a}-\ref{thm:AA_c}  of the Arzelà-Ascoli Theorem \ref{thm:Arzela-Ascoli}.

\begin{proposition}
    \label{prop:trunc_stuff_conservation}
    Suppose that Assumption \ref{ass:Astep3} holds. Let $t\in[0,\infty)$ and $\epsilon\in(0,1)$.
    Then for every $\alpha\in[\min\{-\beta,-\lambda_1 \}-r,1]$, it holds
    \begin{align}
        \label{eq:trunc_stuff_conservation}
        \int_{\sd}|x|^\alpha f_\epsilon(dx,t)\leq  \int_{\sd}|x|^\alpha f_{\epsilon,0}(dx),
    \end{align}
    and, moreover,
    \begin{align}
        \label{eq:trunc_stuff_bounded}
        \int_{\sd}|x|^\alpha f_\epsilon(dx,t)\leq  \int_{\sd}|x|^\alpha f_0(dx)<\infty.
    \end{align}
\end{proposition}

\begin{proof}
    Take $\varphi\in C_c(\sd)$, such that $\varphi\ge0$, and 
    $\varphi(x)=|x|^\alpha$ for all $|x|\in[\epsilon,4/\epsilon]$. 
    Then for all $x,y\in\trunc$
    \begin{align*}
        \zeta_\epsilon(x+y)\varphi(x+y)-\varphi(x)-\varphi(y)&\leq
        |x+y|^\alpha-|x|^\alpha-|y|^\alpha\\ &\leq |x|^\alpha +|y|^\alpha -
        |x|^\alpha-|y|^\alpha\\ &\leq 0,
    \end{align*}
    where we used $0\leq\zeta_\epsilon\leq 1$ in the first step and
    $\alpha\leq1$ in the second.  Then putting $\varphi$ into \eqref{coagEq_regu_measure1} and using \eqref{fVanish} 
    yield the desired result
    \[
        \int_{\sd}x^\alpha f_\epsilon(dx,t)\leq  \int_{\sd}x^\alpha f_\epsilon(dx,0).
    \]
     Then, the 
    uniform bound \eqref{eq:trunc_stuff_bounded} follows from
    \eqref{eq:trunc_stuff_conservation}, $f_{\epsilon,0}(\cdot)=f_0(\cdot\cap
    [\epsilon,2/\epsilon])$ and \eqref{eq:ini_cond}.
\end{proof}

\begin{corollary}
    \label{cor:bound_coag_op}
    Suppose that Assumption \ref{ass:Astep3} holds. Then there exists $C\in
    (0,\infty)$, such that
    \begin{align}
        \int_{\sd}\omega(x)f_\epsilon(dx,t)\leq C, \label{eq:unibound}
    \end{align}
    and
    \begin{align}
        \label{eq:coag_op_bounded}
           \int_{\sd}\int_{\sd} K(x,y)f_\epsilon(dx,t)
            f_\epsilon(dy,t)
        \leq C
    \end{align}
    %for all $\varphi\in C^1([0,\infty);C_c(\sd))$ and
    for all $\epsilon\in(0,1)$ and
    for all $t\in[0,\infty)$.
\end{corollary}

\begin{proof}
    The bound \eqref{eq:unibound} follows from \eqref{eq:trunc_stuff_bounded} and \ref{eq:weight}.
    Then \eqref{eq:coag_op_bounded} follows
    from Propositions \ref{prop:weak_sol_well_defined}.
\end{proof}

\subsubsection*{Family maps into a compact metric space}

We will now prove everything needed for the condition \ref{thm:AA_b} of Arzelà-Ascoli by proving that
$\bar{F}\subset C([0,\infty),\mathcal{B})$ for a compact and metrizable subset
$\mathcal{B}\subset \M_b(\sd)$ defined as follows.

\begin{definition}
    We define
    \begin{align}
        \label{eq:codomain_B}
        \mathcal{B} := \{\mu \in \M_b(\sd) \ : \
        |\langle\phi,\mu\rangle|\leq1 \ \forall\phi\in V\},
    \end{align}
    where  $V:=\{\phi\in C_0(\sd) \ : \ \|\phi\|_\infty<\frac{1}{2C}\}$, and
    the constant $C\in(0,\infty)$ is the same as in \eqref{eq:unibound}.
\end{definition}

\begin{proposition}
    \label{prop:codomain_B_nice}
    The set $\mathcal{B}\subset \M_b(\sd)$ is compact and metrizable
    in the weak*-topology.
\end{proposition}

\begin{proof}
    We notice that $V$ is a neighborhood of $0$. Hence, the compactness and metrizability follows
    from Banach-Alaoglu Theorem as $C_0(\sd)$ is a separable 
    Banach space.
\end{proof}

To clarify the statements of the following Propositions we present the following 
Definition and Assumption. 

\begin{definition}
    We define for each $\phi\in C_0(\sd)$ a pseudo-norm
    \[
        p_\phi :\M_b(\sd)\to[0,\infty),\ p_\phi(\mu)=|\langle\phi,\mu\rangle|.
    \]
    We define $\mathcal{P}$ to be the family of these pseudo-norms,
    \[
        \mathcal{P} := \{p_\phi \ : \ \phi\in C_0(\sd)\}.
    \]
\end{definition}

\begin{assumption}
    \label{ass:compacta}
    Let $\M_b(\sd)$ be equipped with the uniformity $\mathcal{V}$ generated by the
    family $\mathcal{P}$ of pseudo-norms, and $C([0,\infty);\M_{b}(\sd))$ 
    with the uniformity of uniform convergence on compact sets. The topologies
    we consider in $\M_b(\sd)$ and $C([0,\infty);\M_{b}(\sd))$ are the ones 
    induced by their uniformities, respectively, the weak*-topology and the topology
    of uniform convergence on compact sets.
\end{assumption}

We are now ready to prove the above claimed result.
\begin{proposition}
    \label{prop:family_maps_to_metric_space}
    Suppose that Assumptions \ref{ass:Astep3} and \ref{ass:compacta} hold.
    Then $\bar{F}\subset C([0,\infty);\mathcal{B})$.
\end{proposition}

\begin{proof}
Fix $t\in[0,\infty)$ and take any $\bar{f}\in\bar{F}$ and $\phi\in C_0(\sd)$.
Then by Proposition \ref{prop:handling_closure} there exists $f_\phi \in F$,
such that
\[
    \left| \int_{\sd} \phi(x)\left[ \bar{f}(dx,t)-f_\phi(dx,t) \right] \right| < \frac{1}{2}.
\]

Then
\begin{align*}
    \left| \int_{\sd} \phi(x)\bar{f}(dx,t)\right|
    &\leq \left| \int_{\sd} \phi(x)f_\phi(dx,t)\right|
    + \left| \int_{\sd} \phi(x)\left[ \bar{f}(dx,t)-f_\phi(dx,t) \right] \right|\\
    & \leq C\|\phi\|_\infty+\frac{1}{2},
\end{align*}
where in the last step we used Corollary \ref{cor:bound_coag_op}. Thereby,
\[
\left| \int_{\sd} \phi(x)\bar{f}(dx,t)\right| \leq 1
\]
for all $\bar{f}\in \bar{F}$ and all
$\phi\in\{\phi\in C_0(\sd) \ : \ \|\phi\|_\infty<\frac{1}{2C}\}=V$.
Hence, $\bar{f}(\cdot,t)\in\mathcal{B}$.
\end{proof}

\subsubsection*{Difference bound for equicontinuity}
In order to prove the equicontinuity condition \ref{thm:AA_c} of Arzelà-Ascoli for the closure
of the regularized family $F$, we are going to need the following result.
\begin{proposition}
    \label{prop:difference_bound_for_equicontinuity}
    Suppose that Assumptions \ref{ass:Astep3} and \ref{ass:compacta} hold.
    Then there exists $C\in(0,\infty)$, such that
    \begin{align}
        \left|
        \int_{\sd}\phi(x)
        \left[
            f_\epsilon(dx,t_2)-f_\epsilon(dx,t_1)
        \right]
        \right|
        \leq C \|\phi\|_\infty|t_2-t_1| \label{eq:difference_bound_for_equicontinuity}
    \end{align}
    for all $\epsilon\in(0,1)$, for all $\phi\in C_0(\sd)$ and for all
    $t_1,t_2\in [0,\infty)$.
\end{proposition}

\begin{proof}
    Fix $\epsilon\in(0,1)$, $\phi\in C_0(\sd)$ and $t_1,t_2\in [0,\infty)$.
    Due to the symmetry $t_1\leftrightarrow t_2$ in
    \eqref{eq:difference_bound_for_equicontinuity}, we may assume without loss
    of generality that $t_1\leq t_2$. The time integrated regularized coagulation equation
    \eqref{coagEq_regu_measure1} with the time independent test function
    $\phi$ implies
    \begin{align}
        \label{eq:int_trunc_lol_lol}
        &\int_{\sd}\phi(x) f_\epsilon(dx,T) =
        \int_{\sd}\phi(x) f_\epsilon(dx,0)\nonumber\\
        &+\frac{1}{2}\int_0^T\int_{\sd}\int_{\sd}
        K(x,y)f_\epsilon(dx,t)f_\epsilon(dy,t)
        \left[
            \zeta_\epsilon(x+y)\phi(x+y)-\phi(x)-\phi(y)
        \right]
        dt.
    \end{align}
    Recall that $0\leq\zeta_\epsilon\leq1$. Subtracting \eqref{eq:int_trunc_lol_lol} at $T=t_2$ from \eqref{eq:int_trunc_lol_lol} at $T=t_1$ implies
    \begin{align}
        \nonumber
        \left|\int_{\sd}\phi(x)
        \left[f_\epsilon(dx,t_2)-f_\epsilon(dx,t_1)\right]\right| 
        %\\
        %%
        %\nonumber
        %&\left|\frac{1}{2}\int_{t_1}^{t_2}\int_{\sd}\int_{\sd}
        %K(x,y)f_\epsilon(dx,s)f_\epsilon(dy,s)
        %\left[\zeta_\epsilon(x+y)\phi(x+y)-\phi(x)-\phi(y)\right]ds\right|
        %\\
        %%
        %\nonumber
        %\leq&
        %\frac{1}{2}\int_{t_1}^{t_2}\int_{\sd}\int_{\sd}
        %K(x,y)f_\epsilon(dx,s)f_\epsilon(dy,s)
        %\left|\zeta_\epsilon(x+y)\phi(x+y)-\phi(x)-\phi(y)\right|ds
        %\\
        &\leq \frac{3}{2} \|\phi\|_\infty
        \int_{t_1}^{t_2}\int_{\sd}\int_{\sd}
        K(x,y)f_\epsilon(dx,s)f_\epsilon(dy,s)ds
        \\
        \label{eq:thing11}
        &\leq C\|\phi\|_\infty|t_2-t_1|,
    \end{align}
    where Corollary \ref{cor:bound_coag_op} was used in the last step.
\end{proof}

\subsubsection*{Closure of the family is compact}

We are now ready to prove that the closure of $F$ is compact.

\begin{proposition}
    \label{prop:compactness}
    Suppose that Assumptions \ref{ass:Astep3} and \ref{ass:compacta} hold.
    Then $\bar{F}$ is a compact subset of $C([0,\infty);\M_{b}(\sd))$.
\end{proposition}

\begin{proof}
    Since the closure $\bar{F}$ is
    closed by definition, it satisfies the condition \ref{thm:AA_a} of Arzelà-Ascoli
    \ref{thm:Arzela-Ascoli}.  The
    weak*-topology of $\M_b(\sd)\cong C_0(\sd)^*$ is same as the
    topology induced by the uniformity $\mathcal{V}$ of $\M_b(\sd)$. Hence,
    condition \ref{thm:AA_b}  is satisfied by Propositions \ref{prop:codomain_B_nice} and
    \ref{prop:family_maps_to_metric_space}.

    It remains to prove that $\bar{F}$ is equicontinuous, i.e., the condition \ref{thm:AA_c}  
    of  Arzelà-Ascoli \ref{thm:Arzela-Ascoli}. Take any $t\in[0,\infty)$ and
    $V\in\mathcal{V}$. Then by the  
    definition of the uniformity of $C([0,\infty),\M_b(\sd))$, there exist
    finite number of test functions $\phi_1,\phi_2,\dots,\phi_n\in C_0(\sd)$
    and $r_1,r_2,\dots,r_n\in(0,\infty)$, such that
    \begin{align}
        \label{eq:uniformity_base_set}
        B:=\bigcap_{j=1}^n \mathbb{B}_{p_{\phi_j},r_j}\subset V.
    \end{align}
    Let $\delta>0$ and $U_\delta=(t-\delta,t+\delta)\cap[0,\infty)$. Since
    $U_\delta$ is open in $[0,\infty)$, by \eqref{eq:uniformity_base_set} and
    definition of equicontinuity, 
    it remains to be proven that there exists $\delta>0$, such that $\bar{f}(U_\delta)\subset
    B[\bar{f}(\cdot,t)]$ for all $\bar{f}\in\bar{F}$,
    where
    \begin{align*}
        B[\bar{f}(\cdot,t)] &= \bigcap_{j=1}^n \{\mu \in \M(\sd) \ : \
        |\langle\phi_j,\mu-\bar{f}(\cdot,t)\rangle|<r_j\}, \text{ and}\\
        \bar{f}(U_\delta) &= \{\bar{f}(\cdot,s) \ : \ s\in U_\delta \}.
    \end{align*}
    Take $\bar{f}\in\bar{F}$ and $s\in U_\delta$. Let $j\in\{1,2,\dots n\}$.
    By Proposition \ref{prop:handling_closure} there exists
    $f \in F$, such that
    \[
        \left|
            \int_{\sd} \phi_j(x)
                \left[
                    \bar{f}(dx,s)-f(dx,s)
                \right]
        \right| < \delta
    \]
    for all $s\in [0,t+\delta]$.

    Then by telescoping and triangle inequality, for all $s\in U_\delta \subset
    [0,t+\delta]$ there holds
    \begin{align*}
        |\langle
            \phi_j,\bar{f}(\cdot,s)-\bar{f}(\cdot,t)
        \rangle|
        &\leq
        \left|
            \int_{\sd}\phi_j(x)
            \left[
                \bar{f}(dx,s)-f(dx,s)
            \right]
        \right|
        +
        \left|
            \int_{\sd}\phi_j(x)
            \left[
                f(dx,s)-f(dx,t)
            \right]
        \right|
        \\
        &+
        \left|
            \int_{\sd}\phi_j(x)
            \left[
                \bar{f}(dx,t)-f(dx,t)
            \right]
        \right|
        \\
        &\leq 2\delta + C\|\phi_j\|_\infty2\delta,
    \end{align*}
    where in the last step we used also Proposition
    \ref{prop:difference_bound_for_equicontinuity} and
    $\sup\{|s-t| \ : \ s,t\in U_\delta \}\leq 2\delta$.
    We are free to choose $\delta$ so that
    \[
     2\delta
     \left(
        1+ C \max\{\|\phi_1\|_\infty,\dots,\|\phi_n\|_\infty\}
     \right)
     <\min\{r_1,\dots,r_n\}
    \]
    holds. Hence, $\bar{f}(U_\delta)\subset B[\bar{f}(\cdot,t)]$ for all
    $\bar{f}\in\bar{F}$. This proves the equicontinuity of $\bar{F}$.

    The compactness of  $\bar{F}$ follows  now from the Arzelà-Ascoli Theorem
     \ref{thm:Arzela-Ascoli}.
 \end{proof}

 \subsubsection*{The family has a convergent subsequence}

 We are finally ready to finish the \ref{step:3} by proving the following 
 Proposition \ref{prop:subsequence_of_family}. 

\begin{proposition}
    \label{prop:subsequence_of_family} 

    Suppose that Assumptions \ref{ass:Astep3} and \ref{ass:compacta} hold. 
    Then there exist $\bar{f}\in \bar{F}\cap C([0,\infty);\M_{+,b}(\sd))$ and
    a strictly decreasing sequence  $(\epsilon_n)_{n=1}^\infty\subset (0,1)$,
    such that $\lim_{n\to\infty} \epsilon_n=0$, and $(\omega
    f_{\epsilon_n})_{n=1}^\infty\subset{F}$ converges to $\bar{f}$ in
    $C([0,\infty);\M_{b}(\sd))$. Moreover,
    \begin{align}
        \label{eq:family_converges}
        \lim_{n\to\infty} \sup_{t\in[0,T]} 
        \left| 
            \int_{\sd} \phi(x)
            \left[
                \omega(x)f_{\epsilon_n}(dx,t)-\bar{f}(dx,t)
            \right]
        \right|=0, 
    \end{align}
    for every $T\ge0$ and $\phi \in C_0(\sd)$.
\end{proposition}

\begin{proofof}[Proof of Proposition \ref{prop:subsequence_of_family}]
    By \ref{prop:family_maps_to_metric_space} and \ref{prop:compactness},  
    the closure $\bar{F}$ is compact in 
    $C([0,\infty);\mathcal{B})$ when the uniformity of
    $C([0,\infty);\mathcal{B})$ is induced from $C([0,\infty);\M_b(\sd))$.
    The topology of $\mathcal{B}$ is metrizable by Proposition
    \ref{prop:codomain_B_nice}.  We obtain sequential
    compactness of $\bar{F}$ from Proposition \ref{prop:sequentally_compactness}.
    Then there exist $\bar{f}\in \bar{F}$ and a subsequence $(\omega
    f_{\epsilon_n})_{n=1}^\infty$ of $(\omega f_{1/n})_{n=1}^\infty\subset F$,
    such that $(\omega f_{\epsilon_n})_{n=1}^\infty$ converges to $\bar{f}$.
    Therefore by the definition of the uniformity of uniform convergence on compact sets, for any $r>0$, $T\ge0$ and $\phi
    \in C_0(\sd)$, there exists $N>0$, such that for all $n>N$ we have
    \begin{align*}
        &\omega f_{\epsilon_n} \in \{\bar{g} \in \bar{F} \ | \
            (\bar{f}(\cdot,t),\bar{g}(\cdot,t))\in \mathbb{B}_{p_\phi,r} \ \forall t\in
        [0,T]\}
        \\
        &\iff \sup_{t\in[0,T]} \left| \int_{\sd}
        \phi(x)\left[w(x)f_{\epsilon_n}(dx,t)-\bar{f}(dx,t)\right]\right|<r.
    \end{align*}
    This proves \eqref{eq:family_converges}.
    Finally, $(\omega\fen)_{n=1}^\infty\subset
    C([0,\infty);\M_{+,b}(\sd))$ together with \eqref{eq:family_converges} imply
    for each $t\ge0$ that $\langle \phi, \bar{f}(\cdot,t)\rangle\ge0$ for every
    $\phi\in C_0(\sd)$ taking only positive values. Then the Riesz representation
    Theorem guarantees that $\bar{f}\in
    C([0,\infty);\M_{+,b}(\sd))$.
\end{proofof}

\subsection{Step 4: the limiting measure-valued function satisfies the coagulation equation}

In this subsection we will do the \ref{step:4}, followed by the collection of
all the pieces in order to prove the existence Theorem \ref{thm:existence}. 

We start by giving the following Assumption \ref{ass:candidate} which gives 
us the candidate solution.

\begin{assumption}
    \label{ass:candidate}
    Suppose that Assumptions \ref{ass:Astep3} and \ref{ass:compacta} hold.
    Let $(\omega f_{\epsilon_n})_{n=1}^\infty\subset
    C([0,\infty);\M_{+,b}(\sd))$ and $\bar{f}\in C([0,\infty);\M_{+,b}(\sd))$
    be the converging sequence and its limit, respectively, obtained from
    Proposition \ref{prop:subsequence_of_family}.  We denote our candidate
    solution $f:[0,\infty)\to\M_+(\sd)$ to Theorem \ref{thm:existence}
    by $f:=\frac{1}{\omega}\bar{f}$.
\end{assumption}

\subsubsection*{Initial condition}

Next we prove that our candidate satisfies the initial condition
\eqref{eq:agrees_initial}. 

\begin{proposition}
    \label{prop:candidate_initial}
    Suppose that Assumption \ref{ass:candidate} holds. Then $f(\cdot,0) = f_0(\cdot)$.
\end{proposition}

\begin{proof}
    Take any $\varphi\in C_c(\sd)$. Then by telescoping and triangle inequality we
    obtain
    \begin{align*}
        \left| \int_{\sd} \varphi(x)\left[ f(dx,0)-f_0(dx) \right]  \right|
        \leq&
        \left| \int_{\sd} \varphi(x)\left[ f_{\epsilon_n}(dx,0)-f_0(dx) \right]  \right|
        \\
        &+\left| \int_{\sd} \varphi(x)\left[ f_{\epsilon_n}(dx,0)-f(dx,0) \right]  \right|
    \end{align*}

    Recall that $f_{\epsilon}(\cdot,0)=f_0(\cdot\cap \trunc)$.
    Continuity of $\omega$ and $\omega(x)\ne0$ for all $x\in\sd$ imply
    $\frac{\varphi}{\omega}\in C_c(\sd)$. Then by \eqref{eq:family_converges} and
    the fact that $\varphi$ has a compact support we obtain from the above inequality
    at the limit $n\to\infty$ that
    \begin{align*}
        \left| \int_{\sd} \varphi(x)\left[ f(dx,0)-f_0(dx) \right]  \right|=0.
    \end{align*}
    Hence, the result follows from the Riesz Representation Theorem.
\end{proof}

\subsubsection*{Limit of the regularized coagulation equation}
%\subsubsection*{Moment estimates}

We want to prove that each term in the regularized coagulation equation
\eqref{coagEq_regu_measure2} converges to a corresponding term in the
regularized coagulation equation \eqref{eq:thesis_weak}. In order to take these
limits we need the following Proposition \ref{prop:candidate_bounds} to control
needed moments.

\begin{proposition} \label{prop:candidate_bounds} Suppose that Assumption
    \ref{ass:candidate} holds. Then for any
    $\alpha\in[\min(-\beta,-\lambda _2)-r,1]$ and $t\ge0$ there holds
    \begin{align}
        \label{eq:candidate_bounded1}
        \int_{\sd} |x|^\alpha f(dx,t) \leq \int_{\sd} |x|^\alpha f_0(dx),
    \end{align}
    and
    \begin{align}
        \label{eq:candidate_bounded2}
        \int_{\sd} \omega(x)f(dx,t) \leq C,
    \end{align}
    for some finite constant $C>0$ independent of $t$.
\end{proposition}

\begin{proof}
    For each $k\in\N$ take $\varphi_k\in C_c(\sd)$, such that
    $0\leq\varphi_k\leq1$ and $\varphi_k(x)=1$ for every $|x|\in[1/k,k]$. Then
    for any $n,k\in\N$ there holds
    \begin{align*}
        \int_{\{1/k\leq|x|\leq k\}} |x|^\alpha f(dx,t) \leq& \int_{\sd}
        \varphi_k(x) |x|^\alpha
        \left[
            f(dx,t)-f_{\epsilon_n}(dx,t)
        \right]
        +\int_{\sd} \varphi_k(x)|x|^\alpha f_{\epsilon_n}(dx,t)
        \\
        \leq& \int_{\sd}\varphi_k(x)|x|^\alpha\omega(x)^{-1}\omega(x)
        \left[
            f(dx,t)-f_{\epsilon_n}(dx,t)
        \right]
        +\int_{\sd} |x|^\alpha f_0(dx),
    \end{align*}
    where in the last step we used $\varphi_k\leq1$ and Proposition
    \ref{prop:trunc_stuff_conservation}. Note that $
    x\mapsto\varphi_k(x)|x|^\alpha\omega(x)^{-1}\in C_c(\sd)$. Then by taking
    the limit $n\to\infty$ Proposition \ref{prop:subsequence_of_family} yields
    \[
        \int_{\sd} \mathbb{1}_{\{1/k\leq|x|\leq k\}}|x|^\alpha f(dx,t) \leq \int_{\sd}
        |x|^\alpha f_0(dx).
    \]
    Hence, \eqref{eq:candidate_bounded1} follows from Monotone Convergence Theorem
    when letting $k\to\infty$.
    Finally, \eqref{eq:candidate_bounded1}, \eqref{eq:ini_cond} and
    \eqref{eq:weight} imply
    \eqref{eq:candidate_bounded2}.
\end{proof}

In the following Propositions \ref{prop:converges_derivative}, \ref{prop:converges_coag_op}
and \ref{prop:converges_coag_op_int} we prove that the terms  
in the regularized coagulation equation
\eqref{coagEq_regu_measure2} converges to corresponding terms in the
regularized coagulation equation \eqref{eq:thesis_weak}.

\begin{proposition}
    \label{prop:converges_derivative}
    Suppose that Assumption \ref{ass:candidate} holds and let $t\ge0$. Then
    \begin{align}
        \label{eq:convergence_derivative1}
        \lim_{n \to \infty} \int_{\sd}\varphi(x)f_{\epsilon_n}(dx,t)
        =\int_{\sd}\varphi(x)f(dx,t)
    \end{align}
    for all $\varphi\in C_c(\sd)$, and
    \begin{align}
        \label{eq:convergence_derivative2}
        \lim_{n \to \infty}
        \int_0^t\int_{\sd}f_{\epsilon_n}(dx,s)\partial_s\varphi(x,s)ds
        =\int_0^t\int_{\sd}f(dx,s)\partial_s\varphi(x,s)ds
    \end{align}
    for all $\varphi\in C^1([0,\infty);C_c(\sd))$.
\end{proposition}

\begin{proof}
    For any $\varphi\in C_c(\sd)$ there holds $\frac{\varphi}{\omega}\in
    C_c(\sd)$, since $\omega$ is continuous and strictly positive. Then
    Proposition \ref{prop:subsequence_of_family} implies the first limit
    \eqref{eq:convergence_derivative1}.
    %\begin{align*}
        %\lim_{n \to \infty}
        %\left|
             %\int_{\sd}\varphi(x)f_{\epsilon_n}(dx,t)
            %-\int_{\sd}\varphi(x)f(dx,t)
        %\right|
        %=&
        %\lim_{n \to \infty}
        %\left|
             %\int_{\sd}\varphi(x)
             %\omega(x)^{-1}
             %\left[
                %\omega(x)f_{\epsilon_n}(dx,t)
                %-\omega(x)f(dx,t)
             %\right]
        %\right|
        %\\
        %=&0.
    %\end{align*}

    Take $\varphi\in C^1([0,\infty);C_c(\sd))$. Then compactness of $[0,t]$
    and continuity of $s\mapsto\partial_s\varphi(\cdot,s)/\omega(\cdot)\in
    C([0,\infty);C_c(\sd))$ and continuity of the norm $\|\cdot\|_\infty$ imply
    $M:=\sup_{s\in[0,t]}\|\partial_s\varphi(\cdot,s)/\omega(\cdot)\|_\infty<\infty$.
    Thereby, for every $s\in[0,t]$ there holds
    \begin{align*}
        \left|
            \int_{\sd} f_{\epsilon_n}(dx,s)\partial_s\varphi(x,s)
        \right|
        \leq M\int_{\sd}\omega(x)f_{\epsilon_n}(dx,s)
        \leq MC,
    \end{align*}
    where in the last step Proposition \ref{prop:candidate_bounds} was used. Hence, the
    second limit \eqref{eq:convergence_derivative2} follows from the first limit
    \eqref{eq:convergence_derivative1} with the Dominated Convergence Theorem.
\end{proof}

\begin{proposition}
    \label{prop:converges_coag_op}
    Suppose that Assumption \ref{ass:candidate} holds. Then for all $\varphi\in
   C_c(\sd)$ and for all $t\ge0$ there holds
    \begin{align}
        \nonumber
        \lim_{n \to \infty}  & 
        \int_{\sd}\int_{\sd} K(x,y)
        \left[ 
            \zeta_{\epsilon_n}(x+y)\varphi(x+y)-\varphi(x)-\varphi(y)
        \right] f_{\epsilon_n}(dx,t)
        f_{\epsilon_n}(dy,t)
        \\
        \label{eq:coag_op_lim}
        =&
        \int_{\sd}\int_{\sd} K(x,y)
        \left[ 
           \varphi(x+y)-\varphi(x)-\varphi(y)
        \right]f(dx,t)
        f(dy,t).
    \end{align}
\end{proposition}

\begin{proof}
    Fix  $\varphi\in C_c(\sd)$ and $t\ge0$.  We start by proving that 
    \begin{align}
        \label{eq:phiy_nlim}
        \lim_{n \to \infty}  & 
        \int_{\sd}\int_{\sd} K(x,y)f_{\epsilon_n}(dx,t)
        f_{\epsilon_n}(dy,t)
            \varphi(y)
        =
        \int_{\sd}\int_{\sd} K(x,y)f(dx,t)
        f(dy,t)
           \varphi(y).
    \end{align}
    Since $\varphi$ has a compact support there exist $b>a>0$, such that
    $\supp\varphi\subset \{a\leq|x|\leq b\}$. For each $m\in\N$ take $\phi_m\in
    C_c(\sd)$, such that $\mathbb{1}_{\{1/m\leq|x|\leq m\}}\leq
    \phi_m\leq\mathbb{1}_{\sd}$.
    Note that $(x,y)\mapsto
    K(x,y)\phi_m(x)\varphi(y)\in C_c(\sd\times\sd)$. Hence, for all $m\in\N$
    the limit \eqref{eq:convergence_derivative1} implies
    \begin{align}
        \label{eq:phiym_nlim}
        \lim_{n \to \infty}  & 
        \int_{\sd}\int_{\sd} K(x,y)f_{\epsilon_n}(dx,t)
        f_{\epsilon_n}(dy,t)
            \phi_m(x)\varphi(y)
        =\int_{\sd}\int_{\sd}
        K(x,y)f(dx,t)
        f(dy,t)
            \phi_m(x)\varphi(y).
    \end{align}
    Take $m\in\N$ large enough, so that $1/m<a$ and $m>b$. Then by the symmetry
    \eqref{eq:condK_cont} and upper bound \eqref{eq:condK_sym2} of $K$, there
    is a constant $C\in(0,\infty)$ independent of $m$, such that there holds
    $K(x,y)\leq C |x|^{\max(\gamma+\lambda) }$ for
    $(x,y)\in\{m<|x|\}\times\{a\leq|y|\leq b\}$ and $K(x,y)\leq
    C |x|^{\min(-\beta,-\lambda_1 )}$ for $(x,y)\in\{|x|<1/m\}\times\{a\leq|y|\leq b\}$. Hence,
    $\mathbb{1}_{\{1/m\leq|x|\leq m\}}\leq \phi_m\leq\mathbb{1}_{\sd}$ implies
    \begin{align}
        \nonumber
        &\left| 
            \int_{\sd}\int_{\sd}K(x,y)
            \varphi(y)
            f_{\epsilon_n}(dx,t)
            f_{\epsilon_n}(dy,t)
            -\int_{\sd}\int_{\sd}K(x,y)
            \phi_m(x)\varphi(y)
            f_{\epsilon_n}(dx,t)
            f_{\epsilon_n}(dy,t)
        \right|  
        \\
        \nonumber
        &\leq
        \left| 
            \int_{\sd}\int_{\{|x|<1/m\}} K(x,y)
            [1-\phi_m(x)]\varphi(y)
            f_{\epsilon_n}(dx,t)
            f_{\epsilon_n}(dy,t)
        \right|
        \\
        \nonumber
        &\quad+
        \left| 
            \int_{\sd}\int_{\{m<|x|\}} K(x,y)
            [1-\phi_m(x)]\varphi(y)
            f_{\epsilon_n}(dx,t)
            f_{\epsilon_n}(dy,t)
        \right|
        \\
        \nonumber
        &\leq
            \int_{\{a\leq|y|\leq b\}}\int_{\{|x|<1/m\}} K(x,y)
            |\varphi(y)|
            f_{\epsilon_n}(dx,t)
            f_{\epsilon_n}(dy,t)
            \\
            \nonumber
        &\quad+
            \int_{\{a\leq|y|\leq b\}}\int_{\{m<|x|\}} K(x,y)
            |\varphi(y)|
            f_{\epsilon_n}(dx,t)
            f_{\epsilon_n}(dy,t)
        \\
        \nonumber
        &\leq 
        C\|\varphi\|_\infty \int_{\{|x|<1/m\}}|x|^{\min(-\beta,-\lambda_1) }f_{\epsilon_n}(dx,t)
        +
        C\|\varphi\|_\infty \int_{\{m<|x|\}}|x|^{\max(\gamma+\lambda) }f_{\epsilon_n}(dx,t)
        \\
        \nonumber
        &\leq
        C\|\varphi\|_\infty
        \left( 
            m^{-r}\int_{\{|x|<1/m\}} |x|^{\min(-\beta,-\lambda_1) -r}f_{\epsilon_n}(dx,t)
            + m^{-r'}\int_{\{m<|x|\}}|x|^{\max(\gamma+\lambda) +r'}f_{\epsilon_n}(dx,t)
        \right) 
        \\
        &\leq
        \nonumber
        C\|\varphi\|_\infty(m^{-r}+m^{-r'})
        \\
        &\to 0 \text{ uniformly in $n$ as $m\to\infty$, }
        \label{eq:phiy_uniform_mlim}
    \end{align}
    where we took $r'\in(0,1-\max(\gamma+\lambda) ]$ and used Proposition
    \ref{prop:trunc_stuff_conservation}. Note that the $r>0$ in the above argument is the same as in the  condition \eqref{eq:ini_cond} for the initial data. Then \eqref{eq:phiym_nlim} and 
    \eqref{eq:phiy_uniform_mlim} imply the following limits
    \begin{align}
        \nonumber
        &\lim_{n \to \infty}   
        \int_{\sd}\int_{\sd} K(x,y)
            \varphi(y)
        f_{\epsilon_n}(dx,t)
        f_{\epsilon_n}(dy,t)
        \\
        \nonumber
        &=
        \lim_{n \to \infty}   
        \lim_{m \to \infty}   
        \int_{\sd}\int_{\sd}
         K(x,y)
            \phi_m(x)\varphi(y)
         f_{\epsilon_n}(dx,t)
        f_{\epsilon_n}(dy,t)
        \\
        \nonumber
        &=
        \lim_{m \to \infty}   
        \lim_{n \to \infty}   
        \int_{\sd}\int_{\sd}
         K(x,y)
            \phi_m(y)\varphi(y)
         f_{\epsilon_n}(dx,t)
        f_{\epsilon_n}(dy,t)
        \\
        \nonumber
        &=
        \lim_{m \to \infty}   
         \int_{\sd}\int_{\sd}
        K(x,y)
           \phi_m(x)\varphi(y)
        f(dx,t)
        f(dy,t)
        \\
        \label{eq:phiy_limit_change}
        &=
        \int_{\sd}\int_{\sd} 
        K(x,y)\varphi(y)
        f(dx,t)
        f(dy,t)
           ,
    \end{align}
    where in the last step we used Monotone Convergence Theorem together with Propositions
    \ref{prop:weak_sol_well_defined} and \ref{prop:candidate_bounds}. This
    proves \eqref{eq:phiy_nlim}. Due to the symmetry $x\leftrightarrow y$ in
    \eqref{eq:coag_op_lim}, this proves also the convergence of the $\varphi(x)$ term in
    \eqref{eq:coag_op_lim}.

    For large enough $n$, we have
    $\zeta_{\epsilon_n}(x+y)=1$ when $x+y\in\supp\varphi\subset\{a\leq|x|\leq b\}$. Thereby,
    it remains to be proven that
    \begin{align}
        \label{eq:phixy_nlim}
        \lim_{n \to \infty}  
        \int_{\sd}\int_{\sd} K(x,y)
            \varphi(x+y)
        f_{\epsilon_n}(dx,t)
        f_{\epsilon_n}(dy,t)
        &=
        \int_{\sd}\int_{\sd}
        K(x,y)
           \varphi(x+y)
        f(dx,t)
        f(dy,t).
    \end{align}
    For $m\in\N$, let $A_m=\{a\leq |x+y|\leq b\}\cap\{1/m\leq |x|\}^2\subset
    \sd\times\sd$. Note that $A_m$ is compact for each $m\in\N$. For each $m\in\N$
    take $\psi_m\in C_c(\sd\times\sd)$, such that
    $\mathbb{1}_{A_m}\leq\psi\leq\mathbb{1}_{\sd\times\sd}$. 
     Hence, for all
    $m\in\N$ the limit \eqref{eq:convergence_derivative1}  implies
    \begin{align} 
        \label{eq:phixym_nlim}
        &\lim_{n \to \infty}
        \int_{\sd}\int_{\sd} 
        K(x,y)
            \psi_m(x,y)\varphi(x+y)
            f_{\epsilon_n}(dx,t)
        f_{\epsilon_n}(dy,t)
        \\
        &=
        \int_{\sd}\int_{\sd} 
        K(x,y)
           \psi_m(x,y)\varphi(x+y)
           f(dx,t)
        f(dy,t).
    \end{align}
    By the upper bound \eqref{eq:condK_sym2} of $K$ and $-\lambda \leq\gamma+\lambda $, there is a constant
    $C\in(0,\infty)$ independent of $m$, such that there holds $K(x,y)\leq C
    x^{\min(-\beta,-\lambda_1) }y^{\min(-\beta,-\lambda_1) }$ for $(x,y)\in\{a\leq |x+y|\leq b\}$ . Hence,
    $\mathbb{1}_{A_m}\leq\psi\leq\mathbb{1}_{\sd\times\sd}$  and $\supp\varphi\subset
    \{a\leq|x| \leq b\}$ imply
    \begin{align}
        \nonumber
        &\Bigg| 
            \int_{\sd}\int_{\sd}  K(x,y)
            \varphi(x+y)
            f_{\epsilon_n}(dx,t)
            f_{\epsilon_n}(dy,t)
            \\
            \nonumber
            &- \int_{\sd}\int_{\sd} K(x,y)
            \psi_m(x,y)\varphi(x+y)
            f_{\epsilon_n}(dx,t)
            f_{\epsilon_n}(dy,t)
        \Bigg|  
        \\
        \nonumber
        &=
        \left| 
            \iint_{\{a\leq |x+y|\leq b\}\setminus\{1/m\leq|x|\}^2} K(x,y)
            (1-\psi_m(x,y))\varphi(x+y)
            f_{\epsilon_n}(dx,t)
            f_{\epsilon_n}(dy,t)
        \right|
        \\
        \nonumber
        &\leq
            \iint_{\{a\leq |x+y|\leq b\}\setminus\{1/m\leq|x|\}^2} K(x,y)
            |\varphi(x+y)|
            f_{\epsilon_n}(dx,t)
            f_{\epsilon_n}(dy,t)
        \\
        \nonumber
        &\leq
        C\|\varphi\|_\infty
        \iint_{\{a\leq |x+y|\leq b\}\setminus\{1/m\leq|x|\}^2} |x|^{\min(-\beta,-\lambda_1) }|y|^{\min(-\beta,-\lambda_1) }f_{\epsilon_n}(dx,t)
        f_{\epsilon_n}(dy,t)
        \\
        \nonumber
        &\leq 
        2C\|\varphi\|_\infty
        \int_{\{|y|<1/m\}}|y|^{\min(-\beta,-\lambda_1) }f_{\epsilon_n}(dy,t)\int_{\{a/2\leq|x|\leq b\}}|x|^{\min(-\beta,-\lambda_1 )}f_{\epsilon_n}(dx,t)
        \\
        \nonumber
        &\leq
        C\|\varphi\|_\infty
        m^{-r}\int_{\{|x|<1/m\}} |x|^{\min(-\beta,-\lambda_1 )-r}f_{\epsilon_n}(dx,t)\int_{\{a/2\leq|x|\leq b\}}|x|^{\min(-\beta,-\lambda_1) }f_{\epsilon_n}(dx,t)
        \\
        &\leq
        \nonumber
        C\|\varphi\|_\infty m^{-r}
        \\
        &\to 0 \text{ uniformly in $n$ as $m\to\infty$, }
        \label{eq:phixy_uniform_mlim}
    \end{align}
    where we used Proposition \ref{prop:trunc_stuff_conservation}. Then equivalent
    reasoning as what was done in \eqref{eq:phiy_limit_change} implies 
    \eqref{eq:phixy_nlim}.
\end{proof}

\begin{proposition}
    \label{prop:converges_coag_op_int}
    Suppose that Assumption \ref{ass:candidate} holds. Then for all $\varphi\in
    C^1([0,\infty);C_c(\sd))$ and for all $t\ge0$, it holds
    \begin{align*}
        \lim_{n \to \infty}  & 
        \int_0^t\int_{\sd}\int_{\sd} K_{\epsilon_n}(x,y)
        \left[ 
            \zeta_{\epsilon_n}(x+y)\varphi(x+y,s)-\varphi(x,s)-\varphi(y,s)
        \right] 
        f_{\epsilon_n}(dx,s)
        f_{\epsilon_n}(dy,s)
        ds
        \\
        =&
        \int_0^t\int_{\sd}\int_{\sd} K(x,y)
        \left[ 
           \varphi(x+y,s)-\varphi(x,s)-\varphi(y,s)
        \right]
        f(dx,s)
        f(dy,s)
        ds.
    \end{align*}
\end{proposition}

\begin{proof}
    This follows from the Dominated Convergence Theorem together with Propositions
    \ref{cor:bound_coag_op} and \ref{prop:converges_coag_op}.
\end{proof}

\subsection{Proof of the existence result}
\label{sec:existence_proof}

\begin{proofof}[Proof of Theorem \ref{thm:existence}]
    We are now ready to put together the four Steps presented in the subsection \ref{sec:plan} and finish the proof.
    \begin{enumerate}[wide=0pt, leftmargin=\parindent, label=Step \arabic*]
        \item 
            For each $\epsilon\in(0,1)$ we are going to regularize the
            existence problem as follows. We define a new initial data
            $f_{\epsilon,0}\in\M_{+,b}(\R_*)$ by restricting the original
            initial data $f_0\in\M_+(\R_*)$ into the compact set $\trunc$, namely,
            $f_{\epsilon,0}(\cdot):=f_0(\cdot\cap\trunc)$. To
            avoid coagulation producing objects larger than $2/\epsilon$, we
            multiply the gain term by a continuous function
            $\zeta_\epsilon:\sd\to[0,1]$, such that $\zeta_\epsilon(x)=1$ for
            $|x| \leq 1/\epsilon$, and $\zeta_\epsilon(x)=0$ for
            $|x|\ge2/\epsilon$. We are now looking for regularized solutions with the 
            original kernel $K$ and regularized initial data $f_{\epsilon,0}$ in the 
            sense of Definition \ref{def:requ_solu}.
        \item 
            In section \ref{sec:regulized} we proved Theorem
            \ref{thm:regularized_existence_uniqueness} from which we obtain for
            each $\epsilon\in(0,1)$ a unique regularized solution $f_\epsilon$ with
            initial data $f_{\epsilon,0}$ in the sense of Definition
            \ref{def:requ_solu}. 
        \item 
            Using the weight $\omega(x)=x^{\min(-\beta,-\lambda_1)}$ for $|x|\leq1$ and
            $\omega(x)=x^{\max(\gamma+\lambda)}$ for $|x|>1$ we obtained in the Proposition 
            \ref{prop:subsequence_of_family} that the closure of the weighted
            family $F=\{\omega\fe\}_{\epsilon\in(0,1)}$ is sequentially compact 
            in the topology of uniform convergence on compact sets. For us this means 
            that there is a subsequence $(\omega\fen)_{n=1}^\infty\subset F$ and 
            a measure-valued function $\bar{f}\in\bar{F}\cap
            C([0,\infty);\M_{+,b}(\R_*))$, such that $\omega\fen\to\bar{f}$ in
            the sense of \eqref{eq:family_converges}.
        \item 
            At this point we have a suitable candidate solution
            $f:=\frac{1}{\omega}\bar{f}$ to the coagulation equation \eqref{B4}
            in the sense of Definition \ref{def:time-dep-sol-thesis}.  The
            candidate $f$ satisfies the condition \eqref{eq:continuity} of
            Definition \ref{def:time-dep-sol-thesis} since $\omega
            f=\bar{f}\in C([0,\infty);\M_b(\R_*))$. We proved the Propositions
            \ref{prop:candidate_initial} and \ref{prop:candidate_bounds} to
            make sure that $f$ agrees with the initial condition $f_0$ and
            satisfies the conditions \eqref{eq:bound_gamma_moment} and
            \eqref{eq:conservation}. By Remark \ref{re:regu_sol} each $\fen$
            satisfies the integrated version \eqref{coagEq_regu_measure1} of
            the regularized coagulation equation \eqref{coagEq_regu_measure2}.
            In Propositions \ref{prop:converges_derivative} and
            \ref{prop:converges_coag_op_int} we proved that each term in
            \eqref{coagEq_regu_measure1} converges to the corresponding term in
            the desired coagulation equation \eqref{eq:thesis_weak} for $f$.
            Thereby, $f$ satisfies the coagulation equation
            \eqref{eq:thesis_weak}.  Thus, our candidate $f$ is a weak solution
            to \eqref{B4} with initial data $f_0$.
    \end{enumerate}
    
\end{proofof}

\subsection{Existence of strong discrete solutions}
\label{sec:discrete}

The following Proposition \ref{prop:solution_stays_discrete} clarifies the 
fact that if the initial data is discrete then the solution stays discrete. 

\begin{proposition}
    \label{prop:solution_stays_discrete}
    Suppose $K$ is as in Definition \eqref{def:time-dep-sol-thesis} and 
    $f:[0,\infty)\to\M_+(\sd)$ is a solution to the coagulation equation \eqref{B4} in the sense 
    of Definition \eqref{def:time-dep-sol-thesis}. If the initial data $f_0=f(\cdot,0)$ satisfies $f_0 = \sum_{\alpha>0} n_0(\alpha)\delta_\alpha$
    for some $n_0:\dsd\to[0,\infty)$, then there exists $n:\dsd\times[0,\infty)\to[0,\infty)$ such that $f(\cdot,t)=\sum_{\alpha>0}n(\alpha,t)\delta_\alpha(\cdot)$ for all $t\ge 0$. Moreover, 
    $t\mapsto n(\alpha,t)$ is continuously differentiable for all $\alpha\in\dsd$, and $n$ satisfies the discrete coagulation equation 
    \eqref{eq:B4_discrete} for all $t\ge0$ and $\alpha\in\dsd$. Also 
    \begin{align}
    \label{eq:some_discrete_things}
       \sum_{\alpha\in\dsd}
       \alpha n(\alpha,t)
       \leq
       \sum_{\alpha\in\dsd}
       \alpha n_0(\alpha),
       \quad 
       \text{and }
       \quad 
       \sup_{s\in[0,t]}
       \sum_{\alpha\in\dsd}
       \omega(\alpha) n(\alpha,s)
       <\infty 
    \end{align}
    for all $t\ge0$. 
\end{proposition}

\begin{proof}
    By Proposition \ref{prop:valid_test_functions} $\phi(x)=|x|\ind(|x| < 1)$ is a valid test function. Then $\phi(x+y)-\phi(x)-\phi(y)\leq 0$ for all $x,y\in\sd$. Since $|\alpha|\ge1$ for all $\alpha\in\dsd\subset\sd$, we have $\int_\sd \phi(x)f_0(dx)=0$.
    Hence, the coagualation equation \eqref{eq:thesis_weak} with $\phi$ as the test function implies
    $ \int_\sd \phi(x)f(dx,t)=0 $ for all $t\ge0$. Thus 
    \begin{align}
        \label{eq:discrete_first_induction_step}
        f(\{|x|<1\},t)=0
    \end{align}
    for all $t\ge0$.
     We want to prove that for each $n\in\N$ there holds
    \begin{align}
        \label{eq:discrete_induction_step}
        \int_\sd \phi_n(x) f(dx,t)=0
    \end{align}
    for all and $t\ge0$, where $\phi_n(x)= \ind(|x|<n)\ind(x\notin\dsd)$.
    Note that by 
   \eqref{eq:discrete_first_induction_step}    we know that \eqref{eq:discrete_induction_step} holds if $n=1$ or for any $n$ if $t=0$ by the Assumption on the initial data.
    We will continue with induction. Let $n\in\N$ and suppose that \eqref{eq:discrete_induction_step} holds with this $n$.
    By Proposition \ref{prop:valid_test_functions}
   and equation \eqref{eq:discrete_first_induction_step}  we know that $\phi_{n+1}$ is a valid test function. Since $\phi_{n+1}(x+y)=0$ for all $x,y\in\sd$ whenever $\phi_n(x)\phi_n(y) =0 $, the coagulation equation \eqref{eq:thesis_weak} with $\phi_{n+1}$ as the test function together with the Assumption that \eqref{eq:discrete_induction_step} holds for $n$ implies that \eqref{eq:discrete_induction_step} holds for $n+1$. Hence, we obtain by induction that \eqref{eq:discrete_induction_step} holds for any $n\in\N$. Thus $f(\sd\setminus \N_0^d,t)=0$ for all $t\ge0$. This means that there exists $n:\dsd\times[0,\infty)\to[0,\infty)$ such that $f(\cdot,t)=\sum_{\alpha>0}n(\alpha,t)\delta_\alpha(\cdot)$ for all $t\ge 0$. Then \eqref{eq:some_discrete_things} follows automatically from conditions \eqref{eq:bound_gamma_moment} and \eqref{eq:conservation}. 

   Take $\alpha \in \dsd$. Let $\varphi\in C_c(\sd)$ be such that $\varphi(\alpha)=1$ and $\varphi(\beta)=0$ for all $\beta\in\dsd,\beta\ne\alpha$. By Definition \ref{def:time-dep-sol-thesis} we know that $t\mapsto \int_\sd \varphi(x) f(dx,t)$ is continuously differentiable. On the other hand $\int_\sd \varphi(x) f(dx,t) = f(\{\alpha\},t))=n(\alpha,t)$. This proves that $t\mapsto n(\alpha,t)$ is continuously differentiable.

   Plugging $\varphi$ into the coagulation equation \eqref{eq:thesis_weak} and taking time derivative from both sides together with $f(\cdot,t)=\sum_{\alpha>0}n(\alpha,t)\delta_\alpha(\cdot)$
   imply 
   \begin{align*}
                \frac{d}{dt}n({\alpha},t)
                =&
                \frac{1}{2}
                \sum_{\alpha'>0}
                \sum_{\beta>0}
                K(\alpha',\beta)n(\alpha',t)n(\beta,t)
                \left[\varphi(\alpha'+\beta)-\varphi(\alpha')-\varphi(\beta)\right]
                \\
                =& 
                \frac{1}{2}
                \sum_{\beta<\alpha}
                K(\alpha-\beta,\beta)n(\alpha-\beta,t)n(\beta,t)
                \\&
                -
                \frac{1}{2}
                n(\alpha,t)
                \sum_{\beta>0}
                K(\alpha,\beta)n(\beta,t)
                -
                \frac{1}{2}
                n(\alpha,t)
                \sum_{\alpha'>0}
                K(\alpha',\alpha)n(\alpha',t),
\end{align*}
where we used the properties of $\varphi$ to obtain the second step. After renaming $\alpha'$ by $\beta$ in the final term and using symmetry of $K$, we have arrived to the discrete coagulation equation \eqref{eq:B4_discrete}.
     
\end{proof}

\begin{proofof}[Proof of Corollary \ref{cor:discrete_existence}]
    This follows automatically from Theorem \ref{thm:existence} and Proposition \ref{prop:solution_stays_discrete}. 
\end{proofof}

%auto-ignore
% final sec 5
\section{Proof of mass-conservation}
\label{sec:mass}

\begin{proofof}[Proof of Theorem \ref{thm:mass_conservation}]
    Fix arbitrary $j\in\{1,2,\dots,d\}$ and $t\ge0$.  It remains to be proven
    that $\int_\sd x_jf(dx,t)=\int_\sd x_jf_0(dx)$. The idea is to take a
    specific sequence of compactly supported test functions $\varphi_k$, which
    approximates $x_j$. Then the result follows from Dominated Convergence
    Theorem after putting $\varphi_k$ into \eqref{eq:thesis_weak} and taking
    the limit $k\to\infty$.

    We start by constructing the test functions $\varphi_k$. Take two smooth functions
    $\eta_1,\eta_2\in C^\infty((0,\infty))$ such that
    $\ind_{[1,\infty)}\leq\eta_1\leq\ind_{[\frac{1}{2},\infty)}$ and
    $\ind_{(0,1]}\leq\eta_2\leq\ind_{(0,2]}$. For each $k\in\N$ let
    $\varphi_k \in C_c^\infty(\sd)$ be defined by
    $\varphi_k(x)=x_j\eta_1(k|x|)\eta_2(\frac{1}{k}|x|)$. Then
    $\varphi_k(x)=x_j$ whenever $\frac{1}{k}\leq |x| \leq k$. Since
    $|\partial_i\varphi_k(x)|\leq
    1+\max_{r\in[1/2,1]}|\eta_1'(r)|+2\max_{r\in[1,2]}|\eta_2'(r)|$ for all $x\in\sd$, $i\in\{1,2,\dots,d\}$ and $k\in\N$, there exists $C>0$ independent of $k$ such that 
    the Lipschitz property  
    $|\varphi_k(x+y)-\varphi_k(x)|\leq C
    |y|$ holds for all $x,y\in \sd$. 

    %For each $k\in\N$,
    %take
    %$\varphi_k\in C_c(\sd;[0,\infty))$ such that $\varphi_k(x)=x_j$ whenever
    %$x\in[0,k]^d\setminus [0,\frac{1}{k}]^d$, and $\varphi_k(x)\leq |x|$ for
    %all $x\in\sd$, and $\varphi_k$ is Lipschitz in the sense that for some
    %$C>0$ independent of $k$ there holds 
    %$|\varphi_k(x+y)-\varphi_k(x)|\leq C
    %|y|$ for all $x,y\in \sd$. 

    Telescoping with $\int_\sd\varphi_k(x)f(dx,t)$ and
    $\int_\sd\varphi_k(x)f_0(dx)$ and using the triangle inequality imply 
    \begin{align}
        \nonumber
        \left| 
            \int_\sd x_jf(dx,t)-\int_\sd x_j f_0(dx)
        \right|  
        \leq&
        \left| 
            \int_\sd x_jf(dx,t)-\int_\sd \varphi_k(x)f(dx,t)
        \right|  
        \\
        \nonumber
        &+
        \left| 
            \int_\sd x_jf_0(dx)-\int_\sd \varphi_k(x)f_0(dx)
        \right|  
        \\
        &+
        \left| 
            \int_\sd \varphi_k(x)f(dx,t)-\int_\sd \varphi_k(x)f_0(dx)
        \right|  
        \label{eq:some_mass_ineq}
    \end{align}

    Note that $\lim_{k \to \infty} \varphi_k(x)=x_j$ for all $x\in\sd$. Hence,
    Dominated Convergence Theorem together with \eqref{CondInVal} imply that the 
    first two terms on the right hand side of \eqref{eq:some_mass_ineq} goes to zero 
    as $k\to\infty$. Thereby, 
    \[
        \left| 
            \int_\sd x_jf(dx,t)-\int_\sd x_j f_0(dx)
        \right|  
        \leq
        \lim_{k \to \infty}
        \left|\int_{\sd}\varphi_k(x)f(dx,t)-\int_{\sd}\varphi_k(x)f_0(dx)\right|.
    \]
    It only remains to be proven that the right hand side of the above inequality 
    is zero.

    Putting the time independent test function $\varphi_k$ into the coagulation
    equation \eqref{eq:thesis_weak} and using the symmetry of the kernel $K$ imply
    \begin{align}
        \nonumber
        &\left|\int_{\sd}\varphi_k(x)f(dx,t)-\int_{\sd}\varphi_k(x)f_0(dx)\right|
        \\
        \label{eq:estimate_mass_conservation}
        &\leq
        \int_0^t\iint_{\{|y|\leq |x|\}}
        K(x,y)
        \left| 
            \varphi_k(x+y)-\varphi_k(x)-\varphi_k(y)
        \right| 
        f(dx,s)f(dy,s)
        ds.
    \end{align}
    Since $\lim_{k \to \infty} \varphi_k(x+y)-\varphi_k(x)-\varphi_k(y)=0$ for
    all $x,y\in\sd$, it is sufficient to prove that we can take the limit
    $k\to\infty$ inside the integrals on the right hand side of
    \eqref{eq:estimate_mass_conservation}. For this we will use Dominated Convergence Theorem. Define {$$-\lambda_0:=-\beta \text{  and  } \gamma_0:=-2\beta,$$} i.e. $\gamma_0+\lambda_0=-\beta$.
     By the upper bound
    \eqref{eq:condK_sym2} of $K$, Lipschitz continuity of $\varphi_k$ and
    $\varphi_k(x)\leq|x|$, there holds 
    \[
        K(x,y)
        \left| 
            \varphi_k(x+y)-\varphi_k(x)-\varphi_k(y)
        \right| 
        \leq 
        C 
        \sum_{l=0}^{1}|x|^{\gamma_l+\lambda _l}|y|^{1-\lambda _l}
    \]
    whenever $|y|\leq\min(1,|x|)$, and 
    \begin{align*}
        K(x,y)
        \left| 
            \varphi_k(x+y)-\varphi_k(x)-\varphi_k(y)
        \right| 
        \leq 
        C |x|^{\gamma_2+\lambda_2}|y|^{1-\lambda_2}
    \end{align*}
    whenever $1\leq |y|\leq|x|$. Then 
    \begin{align}
        \label{eq:estimate_mass_conservation2}
        \sup_{s\in[0,t]}
        \sum_{l=0}^{1}\iint_{\{|y|\leq \min(1,|x|)\}}
        |x|^{\gamma_l+\lambda _l}|y|^{1-\lambda _l} 
        f(dx,s)f(dy,s)
        <
        \infty
    \end{align}
    by \eqref{eq:bound_gamma_moment} and the fact that $|y|^{1-\lambda _l}\leq|y|^{-\lambda _l}$ for
    $|y|\leq1$.

    Suppose that $-\lambda_2\ge0$. Then $|x|^{\lambda_2}|y|^{-\lambda_2}\leq1$ whenever
    $|y|\leq|x|$. Hence, $\gamma_2\leq1$ implies
    $|x|^{\gamma_2+\lambda_2}|y|^{1-\lambda_2}\leq |x|^{\gamma_2}|y|\leq|x||y|$ for all
    $(x,y)\in\{|y|\leq|x|\}\cap\{|y|>1\}$. Then \eqref{eq:conservation} yields
    \begin{align}
        \label{eq:estimate_mass_conservation3}
        \sup_{s\in[0,t]}
        \iint_{\{|y|\leq|x|\}\cap\{|y|>1\}}
        |x|^{\gamma_2+\lambda_2
        }|y|^{1-\lambda_2} 
        f(dx,s)f(dy,s) 
        <\infty.
    \end{align}

    Suppose that $-\lambda_2<0$. Then $|y|^{1-\lambda_2}\leq |y|$ for $|y|>1$. Also
    $|x|^{\gamma_2+\lambda_2}\leq|x|$ for $|x|>1$, since $\gamma_2+\lambda_2\leq1$. Hence,
    \eqref{eq:conservation} imply
    \begin{align}
        \label{eq:estimate_mass_conservation4}
        \sup_{s\in[0,t]}
        \iint_{\{|y|\leq|x|\}\cap\{|y|>1\}}
        |x|^{\gamma_2+\lambda_2}|y|^{1-\lambda_2} 
        f(dx,s)f(dy,s) 
        <\infty.
    \end{align}

    \iffalse
    Then, \eqref{eq:estimate_mass_conservation2}, \eqref{eq:estimate_mass_conservation3} and 
    \eqref{eq:estimate_mass_conservation4} together imply
    \[
        \sup_{s\in[0,t]}
        \sum_{l=1}^{3}\iint_{\{|y|\leq|x|\}}
        |x|^{\gamma_l+\lambda _l}|y|^{1-\lambda _l} 
        f(dx,s)f(dy,s) 
        <\infty.
    \]
    \fi
    Thus, Dominated Convergence Theorem allows us to take the limit
    $k\to\infty$ inside the integrals on the right hand side of
    \eqref{eq:estimate_mass_conservation}.

\end{proofof}

%auto-ignore
% final sec 6

\section{Proof of gelation}
\label{sec:gel}

%Next we see that if for some $k \in \{1,2,3\}$, it holds $\gamma_k > 1$, then any solution fails to conserve mass.  
In order to prove the gelation Theorem \ref{thm:gel} we need the following technical result \ref{prop:posteriori_upper_bound} and its Corollary \ref{cor:L2}.

\begin{proposition}
    \label{prop:posteriori_upper_bound}
    Suppose that $K$ is as in \eqref{eq:condK_cont} and satisfies the upper
        bound \eqref{eq:condK_sym2}, and the lower bound \eqref{eq:lower}with $c_1>0$ and $\ggel,\pgel$ satisfying $\ggel>1$,
        and  $-\pgel\leq\ggel+\pgel\leq 1$. 
    Let $\Phi\colon[0,\infty)\to[0,\infty)$ be such that $\Phi(0)=0$, and $\Phi$ is differentiable almost everywhere with respect to the Lebesgue measure, and the the derivative satisfies 
    \begin{align}
        C_\Phi := \int_0^\infty \Phi'(A)A^{-1/2}dA<\infty. 
    \end{align}
    Then for any solution $f$ to \eqref{B4} in the sense of Definition \ref{def:time-dep-sol-thesis}, and for all $T\ge0$, there holds
    \begin{align}
        \int_T^\infty \left(\int_\sd f(dx,t)|x|^{\ggel/2}\Phi(|x|)
        \right)^2dt\leq \frac{2C_\Phi^2}{c_1}\int_\sd |x|f(dx,t).
    \end{align}
\end{proposition}
\begin{proof}
    This result has been proved for $L^p$-functions and $d=1$ in \cite[Theorem 2.2]{Gelation}. The proof relies on the use of test functions $x\mapsto\min(x,A)$ labeled by $A\ge0$. The exactly same argument works in our case for measure valued solutions and the only modification needed for the multicomponent setting is to use test
    functions $x\mapsto \min(|x|,A)$. Note that by the Proposition \ref{prop:valid_test_functions} the coagulation equation \eqref{eq:thesis_weak} holds for the test functions $x\mapsto\min(|x|,A)$.
\end{proof}

\begin{corollary}
    \label{cor:L2}
    Suppose the assumptions of the above Proposition \ref{prop:posteriori_upper_bound} and $\ggel\in(1,2)$. Then 
    \begin{align}
        \int_T^\infty \left(\int_{\{|x|\ge R\}} |x|f(dx,t)
        \right)^2dt \leq C_{\ggel} R^{1-\ggel}\int_\sd|x|f(dx,T)
    \end{align}
    for all $T\ge0$ and $R>0$, where $C_{\ggel}\in(0,\infty)$ depends on $\ggel$.
\end{corollary}
\begin{proof}
    This follows from Proposition \ref{prop:posteriori_upper_bound}
    with $\Phi(A)=\max(0,A^{1-\ggel/2}-(R/2)^{1-\ggel/2})$.
\end{proof}

We are now ready to give the proof of the gelation Theorem \ref{thm:gel}.

\begin{proofof}[Proof of Theorem \ref{thm:gel}]
    This proof follows the argument given in the proof of \cite[Theorem 2.4]{Dust} for 
    $L^1$ valued solutions. We present the proof in our setting for the readers convenience.
    
    Fix $R>0$ so that 
    \begin{align}
        \int_{\{|x|\leq R\}}|x|f_0(dx)\leq \frac{1}{2} \int_\sd|x|f_0(dx).
    \end{align}
    By Proposition \ref{prop:valid_test_functions}, the equation \eqref{eq:thesis_weak} holds for the test function $\varphi(x,t)=|x|\ind(|x|\leq R)$.  Since $|x+y|\ind(|x+y|\leq R)-|x|\ind(|x|\leq R)-|y|\ind(|y|\leq R)\leq 0$ for all $x,y\in\sd$, the equation \eqref{eq:thesis_weak} gives that $\int_{\{|x|\leq R\}}|x|f(dx,t)\leq\int_{\{|x|\leq R\}}|x|f_0(dx)$ for all $t\ge0$.

    Suppose towards contradiction that $\int_\sd|x|f(dx,t)=\int_\sd|x|f_0(dx)$ for all $t$. Then 
    \begin{align}
        \int_\sd|x|f_0(dx) =& \int_{\{|x|< R\}}|x|f(dx,t)
        + \int_{\{|x|\ge R\}}|x|f(dx,t)
        \\
        \leq& \frac{1}{2} \int_\sd|x|f_0(dx)+ \int_{\{|x|\ge R\}}|x|f(dx,t),
    \end{align}
    and so,
    \begin{align}
        \label{eq:contradiction}
        \frac{1}{2} \int_\sd|x|f_0(dx) \leq \int_{\{|x|\ge R\}}|x|f(dx,t)
    \end{align}
    Corollary \ref{cor:L2} for $T=0$ implies
    \begin{align}
        \int_0^\infty \left( \int_{\{|x|\ge R\}} |x|f(dx,t)
        \right)^2dt <\infty
    \end{align}
    Hence, $\int_0^\infty \left( \frac{1}{2} \int_\sd|x|f_0(dx)
        \right)^2dt<\infty$. This is a contradiction as $\frac{1}{2} \int_\sd|x|f_0(dx)>0$
        is a non zero constant.    
\end{proofof}

\newpage

%auto-ignore
% final sec 7
\section{Localization in mass-conserving solutions}
\label{sec:localization}

The so-called localization property has recently raised some interest in the mathematical study of multicomponent systems \cite{FLNV2, localization, Hoogendijk2024}.  In general terms, it consists in the concentration of the solution along a line asymptotically for large sizes in the multicomponent space $\R_*^d$. In \cite{localization} this property has been formulated  as follows: there is a positive continuous function $a(t)$ such that $\lim_{t\to \infty} a(t)=0$  and
\begin{equation}\label{eq:localization}
\lim_{t \to \infty} \left|\int_{A_t \cap \left\{\left|\frac{x}{|x|}-\frac{m_0}{|m_0|}\right|\leq a(t)\right\} } |x|f(x,t)dx-|m_0| \right|=0
\end{equation}
with $A_t$ the self-similar region defined by $A_t:=\{ x\in \R_\ast^d \ | \ a(t)t^{\frac{1}{1-\gamma'}}\leq |x| \leq a(t)^{-1}t^{\frac{1}{1-\gamma'}}\}$ and $m_0$ the initial mass-vector defined by 
$m_0:=\int_{\R^d_*} xf_0(dx)$. Here $\gamma'\in [0,1)$ is the homogeneity of the upper and lower bound of the underlying kernel, namely, the kernel satisfies 
\begin{equation}
\label{eq:classKernels}
c'_1(|x|^{\gamma'+\lambda'}|y|^{-\lambda'} +  |y|^{\gamma'+\lambda'}|x|^{-\lambda'}) \leq K(x,y) \leq c'_2(|x|^{\gamma'+\lambda'}|y|^{-\lambda'} +  |y|^{\gamma'+\lambda'}|x|^{-\lambda'}),
\end{equation}
 with $\lambda' \in \R$, such that  $ -\lambda' \leq \gamma'+\lambda'$ and $\gamma'+\lambda' \in [0,1)$ and some constants $c'_1,c'_2>0$. Note that if $K$ satisfies \eqref{eq:classKernels}, then $K$ satisfies the conditions of the existence Theorem \ref{thm:existence} and the mass-conservation Theorem \ref{thm:mass_conservation} with $\beta=\lambda'$ and $\gamma_j =  \gamma'$ and $\lambda_j = \lambda'$, for $j=1,2$.

The existence of mass-conserving solutions satisfying the localization property \eqref{eq:localization} is proved in \cite[Theorem 1.1]{localization}. We note that \cite[Theorem 1.1]{localization} is missing an assumption that the kernel needs to be homogeneous, i.e., satisfy 
\begin{equation}
    \label{eq:homogeneous}
    K(\rho x,\rho y) = \rho ^{\gamma'} K(x,y), \quad \rho >0, \ x,y \in \R_*^d.
\end{equation}
The homogeneity condition \eqref{eq:homogeneous} is needed in 
\cite[Theorem 1.1]{localization} to prove that the constructed solution $f$ satisfies for some constants $C_0>0$ and  $k>1$, the moment bound 
\begin{equation}\label{eq:momentBound}
\int_{\R^d_*} |x|^k f(dx,t) \leq C_0t^{\frac{k-1}{1-\gamma'}}, \quad t \geq 1.
\end{equation}
The derivation of the growth bound \eqref{eq:momentBound} was given in \cite[Lemma 5.2]{localization} following a generalization of the one-component case given in \cite{Dust} for $L^1$ solutions. We note that the existence of mass-conserving solutions was only outlined in the proof of \cite[Theorem 1.1]{localization}, which also served as a motivation to prove the existence Theorem \eqref{thm:existence} and the mass-conservation Theorem \eqref{thm:mass_conservation} of the present paper. 
In fact, following the same steps as in \cite[Lemma 5.2]{localization} it is possible to prove that the regularized solution $f_\epsilon$ defined in Section \eqref{sec:regulized} of the present paper satisfies the bound \eqref{eq:momentBound} with $C_0$ independent of $\epsilon$. 
Consequently, under the same conditions as in \cite[Lemma 5.2]{localization}, the solution $f$ we construct in the proof of Theorem \ref{thm:existence} also satisfies the bound \eqref{eq:momentBound}.  
By combining this construction with our mass-conservation Theorem \ref{thm:mass_conservation} and the results in \cite{localization} we then obtain the following result.

\begin{theorem} 
\label{thm:localiation}
Suppose that $K$ satisfies \eqref{eq:condK_cont}, the homogeneity property \eqref{eq:homogeneous} and the lower and upper bounds
    \eqref{eq:classKernels} with the real parameters $\lambda'$ and $\gamma'$ satisfying $ -\lambda' \leq \gamma'+\lambda'$ and $\gamma'+\lambda',\gamma' \in [0,1)$. Suppose
    that  a given initial data $f_0 \in \mathcal{M}_+(\sd)$ satisfies

    \begin{equation}
       \label{eq:localiation_ini_cond}%
        \int_{\sd}\left(|x|^{-\lambda'-r}+|x|^{1+r}\right)f_0(dx) <\infty
    \end{equation}
      for some $r>0$. 
%Let $\gamma'$ and $\lambda'$ be given real numbers satisfying $ -\lambda' \leq \gamma'+\lambda'$, $\gamma'+\lambda' <1$ and $\gamma' <1$.
% Suppose that $K$ and $f_0$ satisfy the conditions of Theorem \ref{thm:existence} with $\beta=\lambda'$ and $\gamma_j =  \gamma'$ and $\lambda_j = \lambda'$, for $j=1,2$.
Then, there is a mass-conserving solution
    $f:[0,\infty)\to\M_+(\sd)$  to \eqref{B4} with the
    initial data $f_0$ in the sense of definition
    \ref{def:time-dep-sol-thesis} satisfying the growth bound \eqref{eq:momentBound} and the localization property \eqref{eq:localization}.
    Moreover, if the initial data satisfies a second moment bound 
    $$\int_{|x|<1} |x|^{-2\lambda'}f_0(dx) + \int_{|x|>1} |x|^{2(\gamma'+\lambda')}f_0(dx) <\infty,$$ 
    then this solution is unique.  
\end{theorem}

The last part of Theorem \ref{thm:localiation} follows directly from the uniqueness of measure-valued solutions obtained by Throm in \cite{ThromUniqueness}.

Recently, a similar localization result was obtained in \cite{Hoogendijk2024} for the product kernel, namely, the multiplicative kernel $K(x,y)= x^TAy$, with $A$ a symmetric and irreducible matrix. This indicates that localization should be expected to hold also for general gelling kernels.

\bigskip

\noindent \textbf{Acknowledgements}
 We would like to thank Jani
Lukkarinen, Juan Velázquez, and Aleksis Vuoksenmaa for the several discussions relevant to the project.
The authors gratefully acknowledge the support of the Academy of Finland, via an Academy project (project No. 339228) and the Finnish Centre of Excellence in Randomness and Structures (project No. 346306). The research of M. A. Ferreira has also been partially funded by the Faculty of Science of University of Helsinki Support Funding 2022, ERC Advanced Grant 741487 and the Centre for Mathematics of the University of Coimbra - UIDB/00324/2020 (funded by the Portuguese Government through FCT/MCTES, \url{https://doi.org/10.54499/CEECINST/00099/2021/CP2783/CT0002}). 
 %and the Starting Package 2024 funded by the Centre National de la Recherche Scientifique.
 The funders had no role in study design, analysis, decision to publish, or preparation of the manuscript.

\noindent\textbf{Compliance with ethical standards} 
\smallskip

\noindent \textbf{Conflict of interest} The authors declare that they have no conflict of interest.

%\include{bibliography}
%\cleardoublepage %fixes the position of bibliography in bookmarks
\phantomsection

\addcontentsline{toc}{section}{\bibname} % This lines adds the bibliography to the ToC
\bibliographystyle{abbrv}
\bibliography{bibliography}

%\nocite{*}
 
\end{document}